%% file: bmps_final.tex
\documentclass{cmslatex}
\usepackage{latexsym, amssymb}

\input{packages.tex}

\renewcommand{\theequation}{\arabic{section}.\arabic{equation}}

\newtheorem{thm}{Theorem}[section]
\newtheorem{prop}[thm]{Proposition}
\newtheorem{lem}[thm]{Lemma}

\newtheorem{defn}[thm]{Definition}

\newtheorem{rem}[thm]{Remark}

\begin{document}

\title{Data Assimilation for hyperbolic conservation laws. \\A Luenberger observer approach based on a kinetic description
\thanks{{Received date / Revised version date}
}}
\author{A.-C.~Boulanger
\thanks {UPMC, CNRS UMR 7598,  Laboratoire Jacques-Louis  Lions, F-75005, Paris and Inria Paris-Rocquencourt EPI ANGE, (anne-celine.boulanger@inria.fr).}
\and P.~Moireau \thanks {Inria Saclay-Ile de France Team M3DISIM.}
\and B.~Perthame \thanks {UPMC, CNRS UMR 7598,  Laboratoire Jacques-Louis  Lions, F-75005, Paris and Inria Paris-Rocquencourt EPI BANG.}
\and J.~Sainte-Marie \thanks {Inria Paris-Rocquencourt EPI ANGE and CETMEF, 2 Boulevard
  Gambetta, 60200 Compi\`egne.}
}



\pagestyle{myheadings} \markboth{Data Assimilation for hyperbolic conservation laws}{A.-C.~Boulanger, P.~Moireau, B.~Perthame, J.~Sainte-Marie}\maketitle

\begin{flushright}
 \textit{To George Papanicolaou, \\a pioneer of so many subjects in PDEs,\\ on the
occasion of his 70th birthday.}
\end{flushright}

\begin{abstract}
Developing robust data assimilation methods for hyperbolic conservation laws is a challenging subject. Those PDEs indeed show no dissipation effects and the input of additional information in the model equations may introduce errors that propagate and create shocks.
We propose a new approach based on the kinetic description of the conservation law. A kinetic equation is a first order partial differential equation in which the advection velocity is a free variable. In certain cases, it is possible to prove that the nonlinear conservation law is equivalent to a linear kinetic equation. Hence, data assimilation is carried out at the kinetic level, using a Luenberger observer also known as the nudging strategy in data assimilation. Assimilation then resumes to the handling of a BGK type equation. The advantage of this framework is that we deal with a single ``linear'' equation instead of a nonlinear system and it is easy to recover the macroscopic variables. 
The study is divided into several steps and essentially based on functional analysis techniques. First we prove the convergence of the model towards the data in case of complete observations in space and time. Second, we analyze the case of partial and noisy observations. To conclude, we validate our method with numerical results on Burgers equation and emphasize the advantages of this method with the more complex Saint-Venant system.
\end{abstract}


\input{introduction.tex}
%

\input{scl.tex}

\input{saintvenant.tex}

\input{discussion.tex}

%
%
%
%
%
%
%
%
%
%
%
%
%
%
\medskip

\bibliographystyle{plain}
\bibliography{bibBMPS}
\end{document}

%% file: packages.tex
\usepackage[utf8]{inputenc}
\usepackage{appendix}
\let\proof\relax

\usepackage{amsbsy,amsfonts,amssymb,amscd,amsthm}
\usepackage{empheq}
\usepackage{bm,bbm}
\usepackage{cases}

\usepackage{ifpdf}
\newif\ifpdf
\ifx\pdfoutput\undefined
\pdffalse 
\else
\pdfoutput=1 
\pdftrue
\fi

\ifpdf
\usepackage[pdftex]{graphicx}
\graphicspath{{Figures/}} \DeclareGraphicsExtensions{.pdf}
\else
\usepackage{graphicx}
\graphicspath{{Figures/}} \DeclareGraphicsExtensions{.eps}
\fi
\usepackage{color}
\usepackage{tikz}
\usepackage{pgfplots}
\usetikzlibrary{backgrounds}

\usepackage[bf,small]{caption}

\newcommand{\Card}{\operatorname{Card}}

\newcommand {\lb} {\lambda}
\newcommand {\Chi} {{\bf \raise 2pt \hbox{$\chi$}} }
\newcommand{\dEalpha}{\frac{\partial^{\alpha}E}{\partial x^{\alpha}}}
\def\be{\begin{equation}}
\def\ee{\end{equation}}
\def\bemult{\begin{multline}}
\def\emul{\end{multline}}
\def\beal{\begin{equation}\begin{aligned}}
\def\eeal{\end{aligned}\end{equation}}
\def\bealn{\begin{equation}\left\{\begin{aligned}}
\def\eealn{\end{aligned}\right.\end{equation}}

\def\beqn{\begin{eqnarray}}
\def\eeqn{\end{eqnarray}}

\def\O{\Omega}
\def\R{\mathbb{R}}

\newcommand {\sgn} { {\rm sgn} }

\newcommand{\argmin}{\mathop{\mathrm{argmin}}}

\newcommand{\1}{\mathbbm{1}}
\newcommand {\f}   {\frac}

\newcommand{\beq}{\begin{equation}}
\newcommand{\eeq}{\end{equation}}
\newcommand{\bea} {\begin{array}{rl}}
\newcommand{\eea} {\end{array}}
\newcommand{\bepa}{\left\{ \begin{array}{l}}
\newcommand{\eepa} {\end{array}\right.}

\providecommand{\norm}[1]{\left\lVert#1\right\rVert}
\providecommand{\abs}[1]{\lvert#1\rvert}
\providecommand{\megaabs}[1]{\left\lvert#1\right\rvert}
 
\def\sgn{\mathop{\rm sgn}\nolimits}

      \makeatletter
      \def\@setcopyright{}
      \def\serieslogo@{}
      \makeatother

\usepackage{color}
\definecolor{myorange}{rgb}{0.9568,0.4941,0.1961}
\definecolor{myred}{rgb}{0.9098,0.1294,0.2078}
\definecolor{myblue}{rgb}{0.0352,0.4981,0.6509}
\definecolor{myhyperblue}{rgb}{0.1607,0.3922,0.9}
\definecolor{mygreen}{rgb}{0.2235,0.6353,0.2588}
\definecolor{mygrey}{rgb}{0.3,0.3,0.3}

\usepackage{hyperref}
\hypersetup{
pdfborder=0 0 0,
colorlinks=true,
linkcolor=mygrey,
citecolor=myhyperblue,
}
\usepackage[normalem]{ulem} 
\usepackage{manfnt}

%% file: introduction.tex

\section{Introduction}
\label{sec:intro}
Data assimilation has become a popular strategy to refine the quality of numerical simulations of complex physical phenomena by taking into account available measurements. In particular numerous works in environmental sciences, but also life sciences --~see \cite{Blum2008,Navon2009,Mallet2012} and references therein --~have used data assimilation to deal with the various sources of error entering and propagating during a simulation and restricting the performance of a numerical prediction. This is particularly the case for hyperbolic systems in general, and especially for conservative systems where, in the absence of dissipation, even small numerical errors are likely to propagate and expand in time \cite{Chapelle2012}. The main idea behind data assimilation is thus to improve the prediction by integrating additional information on the actual specific system considered, with the use of the measurements --~also called observations --~at our disposal.

We consider in this work a general class of conservation laws that may contain non-linearities as typically illustrated by transport equations, Burgers' equation, the shallow-water or Euler systems and we focus especially on the non-viscous configurations. In data assimilation, two types of approaches are available: the variational and the sequential approach which both can be considered in a deterministic or a stochastic functional framework \cite{Bensoussan1971}. We rely, here, on a deterministic description which can be seen as a first step in the handling of more general uncertainties present in the system. On the one hand, the variational approach was popularized by the 4D-\!Var \cite{dimet2010variational}. It consists in minimizing globally a --~usually least-squares based --~cost function representing a compromise during a period of time between (1) the observation discrepancy computed between simulations of the model and the corresponding measurements and (2) an \emph{a priori} confidence in the 
model.
 The cost function gradient is obtained through an adjoint model integration linked to the dynamical model constraint under which the cost function is minimized. Hence, the minimization is classically performed by a gradient descent algorithm involving numerous successive iterations of the combination of the model and adjoint dynamics during a fixed time window. Moreover, as optimal control theory often assumes differentiability, mathematical difficulties arise when viscosity is neglected and in the presence of discontinuities like shocks, see \cite{bardos2002formalism} and \cite{bardos2005data} where a formalism is defined to avoid incorrect computations of gradients (yet keeping the computation of the adjoint model). 
On the other hand, the sequential methods include in the definition of the approximating dynamics a correction term based on the observation discrepancy. The resulting system called observer --~or estimator in the stochastic context --~is then expected to converge in time to the actual system as it incorporates more and more data. Note that in the sequential approach the initial condition is classically not retrieved but only filtered so that we get an estimation at the final time in one forward simulation. In this context, the most popular observer is based on the Kalman filter which can be virtually defined on any model and shown to be equivalent to the variational approach at least for linear systems, ergo its denomination of optimal estimator \cite{Bensoussan1971}. However the Kalman filter, based on dynamic programing principles, is subject to the famous Bellman curse of dimensionality \cite{Bellman1957} which makes it numerically unpractical for partial differential equations (PDE). One way to circumvent this difficulty --~see Section Discussion for possible alternatives --~is to follow 
the path introduced by Luenberger \cite{Luenberger1963,Luenberger1971} and popularized for PDE with the \emph{nudging} appellation \cite{hoke1976initialization,stauffer1990use} to define a sequential observer which converges to the actual system in one forward simulation and in which the correction term remains tractable in practice. The principle is to introduce the simplest possible correction so that the error between the actual observed trajectory and the simulated systems stabilizes rapidly to 0.  This type of method has recently regained interest with applications to a large range of systems like conservation laws, waves, beams, elasticity or fluid-structure interaction \cite{Liu1997,Chapelle2011,Chapelle2012,Moireau2009,bertoglio2011} and with new general improvements, for example in the determination of the nudging coefficients \cite{zou1992optimal,stauffer1993,vidard2003optimal} or for restoring the initial condition \cite{Auroux2005,Ramdani2010,Haine2011}. However, for conservation laws, despite 
some recent effort \cite{auroux2012}, the observer performances remain difficult to analyze in a general context, in particular in the presence of non-linearities. In this study, we propose a new nudging strategy which, for hyperbolic conservation laws admitting a kinetic description, can be thoroughly justified even for non-linear systems.

Let us consider the practical case of systems, with the illustrative example of the Saint-Venant system, which writes, discarding the source term for simplicity:
\begin{subequations}
	\begin{numcases}{}
	 	\partial_t  H + \partial_x ( H u) = 0 \label{massoriginal},\\
		\partial_t  H u + \partial_x ( H u^2 + \frac{g H^2}{2}) = 0\label{momentumoriginal}.
	\end{numcases}
\end{subequations}
This system naturally comes with a single entropy inequality for the energy
\begin{equation*}
E(H,u) = H\frac{u^{2}}{2} + \frac{gH^{2}}{2}, 
\end{equation*}
which satisfies for weak solutions
\be
\partial_t  E + \nabla \cdot \left(  u( E + \frac{g H^2}{2})\right)\leq 0.
\label{energyineqsvt}
\ee
We expound here briefly the strategy of the kinetic Luenberger observer and set forth a crucial point justifying the interest to design such a method: the automatic decrease of energy that it guarantees. Naive data assimilation on the Saint-Venant system can indeed be written as follows:
\begin{subequations}
	\begin{numcases}{}
        \partial_t \hat H + \partial_x (\hat H\hat u) = \lb( H - \hat H), \label{mass}\\
        \partial_t \hat H\hat u + \partial_x (\hat H\hat u^2 + \frac{g\hat H^2}{2}) = 0. \label{momentum}
	\end{numcases}
\end{subequations}
where $H$ represents data and  $\left(\hat H, \hat u\right)$ the observer. The height is indeed usually the easiest variable to measure. With this system we associate the energy inequality
\[
\partial_t \hat E + \nabla \cdot \left( \hat u(\hat E + \frac{g\hat H^2}{2})\right)\leq -\lb( H - \hat H)\left( \frac{\hat u^{2}}{2} - g\hat H\right). 
\]
Unfortunately, the nonlinearity in the right hand side of the inequality prevents control of the energy dissipation. This is where we introduce the kinetic description of the Saint-Venant system. It consists in using the equivalence between the nonlinear system of conservation laws and a linear kinetic equation in a new variable $\xi$, the kinetic velocity. Formally, we have an equivalence between the macroscopic conservation law and its energy inequality and the kinetic equation. The shallow water system without source term \eqref{massoriginal} -\eqref{momentumoriginal} augmented of its entropy inequality \eqref{energyineqsvt} is equivalent to
\be
 \partial_t M(\hat H, \hat u - \xi) + \xi \cdot \nabla_x M(\hat H, \hat u - \xi) = Q(t,x,\xi).
 \label{kin}
\ee
where $M$ shall be regarded as a particle density defined by
\begin{equation}
	M(h, v - \xi) =
	\begin{cases}
 \frac{1}{2\pi g}, & \text{for} |\xi-v| < \sqrt{2gh},  \\
  0, & \text{otherwise. }
	\end{cases}
\label{gibbsform}
\end{equation}
and $Q$ is what is called a collision operator that verifies
\[
\int_{\R_{\xi}}{Q(t,x,\xi)d\xi} = \int_{\R_{\xi}}{\xi Q(t,x,\xi)d\xi} = 0, \quad \int_{\R_{\xi}}{\frac{\xi^2}{2}Q(t,x,\xi)d\xi} \leq 0. 
\]
The proof is done by simple integration in $\xi$ of the first moments of \eqref{kin}.

In comparison to (\eqref{mass} - \eqref{momentum}), the kinetic assimilated system writes (\cite{Perthame2002})
\begin{equation*}
 \partial_t f(\hat H, \hat u - \xi) + \xi \cdot \nabla_x f(\hat H, \hat u - \xi) = \lb\left(f(\hat H, \hat u - \xi) - M( H, u - \xi)\right).
 \label{kinda}
\end{equation*}
where $f$ and $M$ satisfy the form \eqref{gibbsform}. This construction of $M$ requires available observations of both $H$ and $u$. However, we will show in this article that great assimilation strategies can also be carried out with partial observations on water depth or velocity only with the use of an approximated $M$. Thanks to the property of the collision operator, the energy equation, which is classically obtained by computing $\int_{\R_{\xi}}{\frac{\xi^{2}}{2}\eqref{kin}d\xi}$ (see \cite{Perthame2002}) becomes
\[
\partial_t \hat E + \nabla \cdot \left( \hat u(\hat E + \frac{g\hat H^2}{2})\right)\leq \lb( E - \hat E). \nonumber
\]
This ensures relaxation of energy. Unlike the macroscopic nudging \eqref{mass}, the Luenberger observer that we add at the kinetic level does not perturb the stability of the macroscopic system. This simple example shows the nice properties of the kinetic Luenberger observer when used in the kinetic representation. 

This efficacy encourages the development of a theoretical framework in the case of scalar conservation laws as usually made when dealing with hyperbolic PDEs. The kinetic formulation of them is indeed well understood \cite{Perthame2002}. For a scalar conservation law defined as follows
\begin{equation}
	\begin{cases}
 	  \displaystyle \partial_t u + \sum_{i=1}^{d}\partial_{x_{i}}A_{i}(u) = 0, \\
 	  u(t=0,x) = u_{0}(x).
 \label{conservlaw}
	\end{cases}
\end{equation}
with the entropy inequalities, for any convex function $S(\cdot)$
\be
 \partial_t S(u) + \sum_{i=1}^{d}{\partial_{x_i} \eta_i (u)} \leq 0,
 \label{entropyineq}
 \ee
the kinetic description can be written
\begin{equation*}
	\begin{cases}
		\partial_t M(t,x,\xi) +a(\xi)\cdot \nabla_x M(t,x,\xi) =Q(t,x,\xi),\\
		M(0,x,\xi) = M_0(x,\xi). 
	\end{cases}
\end{equation*}
$u$ is the macroscopic variable, $M$ is the density function which is related to $u$ by its first moment:
$$u(t,x) = \int_{\mathbb{R}} M(t,x,\xi)\ d \xi,$$
 $Q$ is the collision operator. Nonlinearities are hidden in the density function $M$, and shocks are confined in the collision operator. But we emphasize the fact that for fixed $\xi$, the kinetic equation can be handled as a linear equation. The equivalence proof is based on simple integration of the kinetic equations. Then, the observer equations at the macroscopic level
\begin{equation*}
	\begin{cases}
		\partial_t \hat u + \nabla_{x}\cdot A(\hat u) = \lb\left(u -\hat u\right),\\
		\hat u(0) = u_0 + \delta_0, 
	\end{cases}
\end{equation*}
are formulated at the kinetic level
\begin{equation*}
	\begin{cases}
 	   \partial_t f + a(\xi)  \cdot \nabla_x f  =
\lb \left(M - f\right),\\
		f(0) = M_0 + \gamma_0. 
	\end{cases}
\end{equation*}
At the macroscopic level, the nonlinearity of the flux $A$ is the main obstacle to easy convergence proofs of $\hat u$ towards $u$. Possible techniques are based on $L^1$ contraction principle in the scalar case, but do not extend to systems. At the kinetic level, the equation gains linearity and we can even have the explicit solution:
\[
f(t,x,\xi) = f_0(x - a(\xi)t,\xi)e^{-\lb t} + \int_0^t{e^{-\lb s}M(t - s,x - a(\xi)s,\xi)ds}, 
\]
as soon as $M(t,x,\xi)$ can be generated from the measurements.

The outline is organized as follows. Section~\ref{sec:scl} exposes the theoretical framework of our novel method on a scalar hyperbolic conservation law. Convergence theorems are given in the case of complete observations, partial observations and noisy observations. Numerical experiments illustrate its efficiency on Burgers' equation in 1D. Afterwards, Section~\ref{sec:saintvenant} introduces the more practical case of the shallow water equations. We emphasize the benefit of our method for the Saint-Venant system as we are able to build an efficient numerical assimilation scheme based on heights observations only. Once again, numerical simulations are carried out on bidimensional analytical solutions and an example in 3D is also provided.

\vspace{0.5cm}
Throughout the paper, $V_{p}$ stands for the Lebesgue space $L^{p}(\R_{x}^{d}\times \R_{\xi})$ and we use the notation $\norm{f}_{L^{p}(\R_{x}^{d}\times\R_{\xi})} = \norm{f}_{V_{p}}$.

%% file: scl.tex

\section{Kinetic Luenberger observer on hyperbolic scalar conservation laws}
\label{sec:scl}
In this section, we present the theoretical framework underlying the novel method we propose. It is based on the particular properties of hyperbolic balance laws and their kinetic formulation.

\subsection{Kinetic formulation}
\label{subsec:kinform}
The kinetic formulation, restricted to scalar conservation laws or systems with many entropies, offers a way to represent full families of entropies by a single equation. An entropy is a conservation laws built on nonlinear functions of the unknown. The idea underlying the kinetic formulation is to parametrize those entropies, and consider the parameters an additional variable. We call it the kinetic velocity $\xi$. After some manipulation, we end up with a linear equation in the best cases (for instance for scalar conservation laws). This linearity enables the use of functional analysis techniques and also to derive simple numerical schemes.
We detail here more precisely what was announced in the introduction in the case of a hyperbolic scalar conservation laws of the form \eqref{conservlaw} where $u(t,x)$ defined for $t \geq 0$ is a real-valued function, $x \in \R^d$.

We know that solutions of \eqref{conservlaw} may exhibit shocks. To be considered, weak solutions shall satisfy the entropy inequality \eqref{entropyineq} in the sense of distribution and for any entropy pair $(S,\eta_i, i=1..d)$, where S is a convex entropy and $\eta_i$ the corresponding entropy flux. We introduce the function $\Chi$ such that for $\xi \in \R$
\begin{equation}
	\Chi (\xi, u) =
	\begin{cases}
    +1, & \text{ for } 0 < \xi < u, \\
    -1, & \text{ for } u< \xi < 0,  \\
     0, & \text{ otherwise}.
	\end{cases}
	\label{chidef}
\end{equation}
This function allows to represent any function $S(u)$ as follows
\[
\int_{\R_{\xi}}{S'(u)\Chi(\xi,u)d\xi} = S(u) - S(0). 
\]
Then, system \eqref{conservlaw} with its family of entropy inequalities \eqref{entropyineq} is equivalent to the kinetic formulation
\be
\partial_t \Chi(\xi, u(t,x)) + a(\xi)\cdot \nabla_x \Chi(\xi, u(t,x)) = \nabla_{\xi}m(t,x,\xi),
\label{kinprelim}
\ee
where we have set
\be
a_i(\cdot) = A_i'(\cdot),
\label{kinfluxexpression1} 
\ee
\be
a(\xi) = \left( a_1(\xi), \cdots, a_d(\xi)\right),
\label{kinfluxexpression2}
\ee 
and for some nonnegative bounded measure $m(t,x,\xi)$, called kinetic entropy defect measure. A proof for this assertion is available in \cite{lions1994kinetic}. Though the converse is difficult to address, one implication can be easily computed by multiplying \eqref{kinprelim} by $S'(\xi)$ and integrating in $\xi$.
\begin{rem}
The kinetic entropy defect measure $m$ is unknown a \textit{priori} and takes account of potential singularities in the solution. It can be seen as the Lagrange multiplier for the constraint \eqref{chidef}, similarly to the pressure term in the incompressible Navier-Stokes equation, which is a multiplier related to the constraint that the flow is divergence free. Three properties are worth notifying here:
\begin{itemize}
 \item $m(t,x,\xi)$ depends on the macroscopic solution $u(t,x)$,
  \item in particular, m(t,x,u(t,x)) = 0,
 \item $m$ vanishes where the solution is smooth.
\end{itemize}
\end{rem}

\subsection{Kinetic level nudging in the scalar case for complete observations}
\label{subsec:nudgingscl}
As previously stated in the introduction, the nudging strategy carried out at the kinetic level offers the advantage of a "linear framework" and can guarantee the decay of the energy in such systems as the shallow water equations. In this section we develop theoretical results for this method. We begin with very ideal situations, among which smooth solutions or complete observations are studied. We are then able to state convergence results in the case of partial observations in space or time. 

Let us consider that the observations derive from the scalar conservation law \eqref{conservlaw}. Following Section~\ref{subsec:kinform}, \eqref{conservlaw} admits the kinetic formulation
\bealn
& \partial_t M(t,x,\xi) + a(\xi)\cdot \nabla_x M(t,x,\xi) = \nabla_{\xi}m(t,x,\xi),\\
& M(0,x,\xi) = \Chi(\xi, u_0(x)),
\label{conservlawkin}
\eealn
where $M(t,x,\xi) = \Chi(\xi, u(t,x))$ satisfies \eqref{chidef} and $a(\xi)$ is defined in \eqref{kinfluxexpression1}-\eqref{kinfluxexpression2}. We recall that $\Chi$ is designed such that
\be
u(t,x) = \int_{\R_{\xi}}{\Chi(\xi, u(t,x))d\xi}.
\label{uintegral}
\ee
We set the study in the case of complete observations for a fixed relaxation parameter $\lb$. Hence, we can design the observer $f(t, x, \xi)$ satisfying the Boltzmann type
equation
\bealn
&\partial_tf(t, x, \xi) + a(\xi) \cdot \nabla_x f(t, x, \xi) = \lb \left( M(t,x,\xi) - f(t,x,\xi)\right), \\
&f(t, x, \xi = 0) = f_{0}(x, \xi).
\label{DA_kintimeeasy}
\eealn
We will also denote the macroscopic observer
$$\widehat u(t,x) = \int_{\R_{\xi}}{f(t,x,\xi)d\xi}.$$
First of all, we precise that \eqref{DA_kintimeeasy} admits a unique distributional solution (\cite{Perthame2002}, section 3.5, Th. 3.5.1):
\begin{lem}
Let $f_{0} \in V_1, M \in  C(0,T),L^{1}( \R_x^d
\times
\R_{\xi}))$ for all $T > 0$. Then, there exists a unique distributional solution 
to
\eqref{DA_kintimeeasy}, $f \in C(\R^+ ;V_1)$ and it satisfies
$$f = f_{0}(x - a(\xi) t)e^{-\lb t} + \lb \int_0^t{e^{-\lb s} M(t-s, \xi,
x - a(\xi) s)ds}.$$
\label{lemexistence}
\end{lem}

\subsubsection{Smooth data}

In the case of smooth solutions, the kinetic defect measure $m$, whose role is to locate the discontinuities, disappears and the simple transport equation remains
\bealn
& \partial_t M(t,x,\xi) + a(\xi)\cdot \nabla_x M(t,x,\xi) = 0,\\
& M(0,x,\xi) = \Chi(\xi, u_0(x)).
\label{conservlawkin0}
\eealn
Therefore, we are able to state a very strong theorem in the simple case of smooth solutions to scalar conservation laws: the kinetic observer converges exponentially fast towards the data and so do the macroscopic variables. The proof is done at the kinetic level by applying Lemma~\ref{lemexistence} to the equation satisfied by the error $f - M$. The macroscopic result follows by integration w.r.t. $\xi$.
\begin{prop}
 Let $ M \in C([0, T] ,V_{1})$ satisfy \eqref{conservlawkin0} and $f \in C(\R^+ ;V_1)$ satisfy \eqref{DA_kintimeeasy}, then
\[
\norm{f -  M}_{V_1}\xrightarrow[t \to +\infty]{}0, 
\]
and the decay is exponential in time with rate $\lb$. This is translated at the macroscopic level as:
\[
\norm{\widehat u -  u}_{L^{1}( R_x^d)}\xrightarrow[t \to +\infty]{}0,  
\]
with the same exponential decay.
\label{thmkineasy}
\end{prop}

\subsubsection{Non-smooth data}

In the situation of non-smooth solutions, the data are assumed to satisfy \eqref{conservlawkin}, \eqref{uintegral} and the kinetic observer $f(t,x,\xi)$ satisfies \eqref{DA_kintimeeasy} since observations are complete in space and time. The presence of the collision term prevents us from obtaining a time convergence as in Proposition~\ref{thmkineasy}. However, we can still obtain a convergence with respect to the nudging parameter $\lb$. In order to show the convergence of $f$ towards $ M$ we introduce another integro-differential equation 
\bealn
&\partial_tf_{\lb}(t, x, \xi) + a(\xi) \cdot \nabla_x f_{\lb}(t, x, \xi) + \lb f_{\lb}(t, x, \xi) = \lb M_{\lb},\\
&f_{\lb}(0, x, \xi) = \Chi(\xi, u_0(x)),
\label{DA_kintime2}
\eealn
where $M_{\lb}(t, x, \xi) = \Chi(\xi, u_{\lb})$ and $u_{\lb}$ is only defined by the relation $u_{\lb} = \int_{\R_{\xi}}{f_{\lb}d\xi}$. Thus, we claim
\begin{thm}
Let $M\in C([0, T] ,V_1)$ satisfy \eqref{conservlawkin} with $ u_0 \in L^1(\R^d)$ and $f \in C(\R^+ ;V_1)$ satisfy \eqref{DA_kintimeeasy} with $f_0 \in V_1$, then for some modulus of continuity $\omega(T,\cdot)$
\[
\norm{f(T) -  M(T)}_{V_1} \leq \norm{ f_0 - f_{\lb,0}}_{V_1}e^{-\lb T} + \omega(T, \frac{1}{\lb}). 
\]
Moreover, if $ u_0 \in L^{\infty}\cap L^1 \cap BV(\R^d)$, the dependency in $\lb$ can be improved as:
\[
\norm{ f(T) - M(T)}_{V_1} \leq \norm{ f_0 - f_{\lb,0}}_{V_1}e^{-\lb T} + C(d)\norm{a}_{L^{\infty}(-K, K)}\sqrt{\frac{T\norm{u_0}_{BV(\R^d)}}{\lb}}.
\]
\label{thmkinform1}
\end{thm}
\begin{rem}
The rate of convergence in $\sqrt{\frac{T}{\lb}}$ expresses the fact that at fixed $\lb$, the discrepancy between the observations and the model variables tend to increase with time. The further we want to advance in time, the bigger we need to choose $\lb$. Notice also that in this situation the optimal $\lb$ is the largest possible as it is often the case with complete observations.
\end{rem}
\begin{proof}[Proof of Theorem~\ref{thmkinform1}]
From the triangular inequality we decompose
\begin{equation*}
\norm{f(T) -  M(T)}_{V_1} \leq \norm{f(T) -
f_{\lb}(T)}_{V_1} + \norm{f_{\lb}(T) -
 M(T)}_{V_1}.
\label{triangularineq}
\end{equation*}
On the one hand, in \cite{Perthame2002}, the author studies some properties of $f_{\lb}$ as $\lb$ tends to infinity, also called the hydrodynamic limit. They prove (Section 4.6, Thm 4.6.1):
\begin{itemize}
 \item if $ u_0 \in L^{\infty}\cap L^1 \cap BV(\R^d)$, and setting $K = \norm{u_0}_{L^{\infty}}$,
\begin{align*}
	\norm{u_{\lb}(T) -  u(T)}_{L^1(\R^d)} &\leq \norm{f_{\lb}(T) -  M(T)}_{V_1},  \\
	&\leq C(d)\norm{a}_{L^{\infty}(-K, K)}\sqrt{\frac{T\norm{u_0}_{BV(\R^d)}}{\lb}}; 
\end{align*}
 \item if only $u_0 \in L^1(\R^d)$ is assumed, we have for some modulus of continuity $\omega(T,\cdot)$: 
$$\norm{u_{\lb}(T) -  u(T)}_{L^{1}( \R^d)} \leq \norm{f_{\lb}(T) -  M(T)}_{L^{1}(\R_x^d \times \R_{\xi})} \leq \omega(T, \frac{1}{\lb}).$$
\end{itemize}
On the other hand, according to \eqref{conservlawkin} and \eqref{DA_kintime2}, $f$ and $f_{\lb}$
satisfy
\begin{equation*}
\partial_{t}(f - f_{\lb})+ a(\xi) \cdot \nabla_x (f - f_{\lb})+ \lb (f - f_{\lb}) = \lb ( M - M_{\lb}).
\label{fflambda1}
\end{equation*}
Therefore, we can use the representation formula of Lemma~\ref{lemexistence} and get
\begin{equation*}
(f - f_{\lb})(t, x, \xi)  = (f - f_{\lb})(0, x, \xi)e^{-\lb t} \\+ \lb
\int_{0}^{t}{e^{-\lb s}( M - M_{\lb})(t-s, \xi, x - a(\xi) s)ds}.
\label{reprform}
\end{equation*}
We end up with
\[
\norm{(f - f_{\lb})(T)}_{V_1} \leq \norm{(f -
f_{\lb})(0)}_{V_1} e^{-\lb T}+ (1 - e^{-\lb T})\norm{ M(T) -
M_{\lb}(T)}_{V_1}. 
\]
Since
$$\norm{ M(T) - M_{\lb}(T)}_{V_1} =  \norm{\Chi_{u} (\xi) - \Chi_{u_{\lb}} (\xi)}_{V_1} = \norm{ u(T) - u_{\lb}(T)}_{L^{1}(\R_{x})},$$ 
we can still use the convergence rates expressed above to upper bound $\norm{ f(T) - f_{\lb}(T)}_{V_1}$ in both cases. This concludes the proof.
\end{proof}

\subsection{Noisy observations in the scalar case}
\label{secnoise}
It is more realistic to consider that observations are noisy. Here we restrict our analysis to additive noise in the observer equation. Appropriate distribution spaces representing this phenomena are the dual Sobolev spaces. The noise is indeed not negligible in the strong $L^{2}$ norm because of its high oscillations, but it is smaller in a weaker $H^{-\alpha}$ norm. This modeling enables us to prove convergence theorems for velocity averaged quantities like $\rho_{\psi}(t,x) = \int_{\R_{\xi}}{f(t,x,\xi)\psi(\xi)d\xi}$ that are known to lie in a better space than the microscopic density $f(t,x,\xi)$. This property results from averaging lemmas \cite{golse1988regularity,golse1985resultat}. In the following, the proofs are based on the same ideas as those developed in \cite{golse1988regularity,golse1985resultat,bouchut1999averaging}.

\subsubsection{Linear coefficient}
\label{subsecnoise1}

We begin with the steady state situation in the case of a linear coefficient for advection. The kinetic variable $\xi$ represents then a multidimensional velocity, and we have $\xi \in \R^{d}$ (as $x \in \R^{d}$). As stated previously, the noise is supposed to be an $\alpha^{th}$ space derivative of an $L^{2}$ function. Let us consider the solution of
\beal
\xi \cdot \nabla_{x}f(x,\xi) + \lb f(x,\xi) = \lb M(x,\xi) + \lb \frac{\partial^{\alpha}E}{\partial x^{\alpha}}, \quad E(t,x,\xi) \in V_{2},
\label{eqnoise10lin}
\eeal
with
\be
\xi \cdot\nabla_{x}M(x,\xi) = Q (x,\xi) \in  V_{2},
\label{eqnoise20lin}
\ee
where $\alpha$ is a multi-index $(\alpha_{1}, \alpha_{2}, \ldots \alpha_{n})$ satisfying $\abs{\alpha} = \alpha_{1} + \alpha_{2}+\ldots+\alpha_{n}$, $$\frac{\partial^{\alpha}E}{\partial x^{\alpha}}=\frac{\partial^{\alpha}E}{\partial x_{1}^{\alpha_{1}}x_{2}^{\alpha_{2}}\ldots x_{n}^{\alpha_{n}}}.$$
\begin{rem}
The noise on the data in a real situation might rather be located on the macroscopic variables. It could be more realistic to translate this macroscopic noise at the kinetic level. However, this transcription is not straightforward and remains to be done. For now, we limit ourselves to the case where we introduce the noise directly on the kinetic observer equation. In other words, the errors introduced by the reconstruction of $M(t,x,\xi)$ from the macroscopic data are considered at their first order.
\label{remnoise}
\end{rem}
We denote $g(x,\xi) = f(x,\xi) - M(x,\xi)$ the error between the estimator and observations. Subtracting \eqref{eqnoise10lin} and \eqref{eqnoise20lin}, the equation for $g$ is
\be
\xi\cdot \nabla_{x}g(x,\xi) + \lb g(x,\xi) =  \lb \frac{\partial^{\alpha}E}{\partial x^{\alpha}} - Q (x,\xi),
\label{eqnoise30lin}
\ee
with an explicit solution (Lemma~\ref{lemexistence})
\[
g(x, \xi) = -\frac{1}{\lb}\int_{0}^{+\infty}{\left(Q(x - \frac{\xi}{\lb}s,\xi) + \lb \frac{\partial^{\alpha}E}{\partial x^{\alpha}}(x - \frac{\xi}{\lb}s,\xi)\right)e^{-\lb s}ds}, 
\]
leading to the upper bound of the error in the $L^{2}$ norm
\be
\norm{g}_{V_2} \leq \frac{\norm{Q}_{V_2}}{\lb} + \norm{\dEalpha}_{V_2}.
\label{l2normerrorlin}
\ee
The nudging coefficient tends to filter the collision term. Nevertheless, it has no effect on the noise. It is then to be expected that in a usual norm as $L^{2}$, we do not get rid of noisy terms. But it is also worth noticing, that this noise is not amplified. This will be all the more important in the evolutionary case in the next section, as the effects of noise usually tend to add up with time since it enters the equations as a source term.

By means of Fourier analysis, we can obtain a better upper bound in a stronger norm for $g$, but considering a weaker norm for the noise. Let us denote by $\widehat g (k,\xi)$ the space Fourier transform of $g$
\[
\widehat g(k,\xi) = \int_{\R^{N}}{g(x,\xi)e^{ik \cdot x}dx}. 
\]
$\widehat E$ and $\widehat Q$ are defined in the same manner. Moreover, we will work on the averaged quantity 
\[
\rho_{\psi}(t,x) = \int_{\R_{\xi}}{g(t,x,\xi)\psi(\xi)d\xi},
\]
where $\psi$ is a smooth compactly supported function. We notice that
\[
\widehat \rho_{\psi}(t,k) = \int_{\R_{\xi}}{\widehat g(t,k,\xi)\psi(\xi)d\xi}. 
\]
 We claim that the averaged quantity $\rho_{\psi}$ lies in a better space than $L^{2}(\left[0,T\right], L^{2}(\R_{x}))$, and its homogeneous seminorm tends to zero when $\norm{E}_{V_{2}}$ becomes smaller. In the following theorem, $\norm{\rho}_{ \dot H^{s}(\R_{k}^{d})}$ represents the homogeneous Sobolev norm of $\rho$.
\begin{thm}
\label{thmnoise1lin}
Let $M(x,\xi) \in V_{2}$ be solution of \eqref{eqnoise20lin} and $f(x,\xi) \in V_{2}$ be solution of the noisy data assimilation problem \eqref{eqnoise10lin} for $\abs{\alpha} < \frac{1}{2}$. Then, there exists $\gamma$, $0 < \gamma < 1 - 2\abs{\alpha}$ and constants $K_1(\alpha, \gamma), K_2(\alpha, \gamma)$ such that
\beal
\norm{ \widehat \rho_{\psi}}_{\dot H^{\gamma/2}(\R_{x})}^2 \leq K_1 \norm{E}_{V_2}^{\frac{2 - \gamma}{1 + \alpha}}  \norm{Q}_{V_2}^{\frac{2\alpha + \gamma}{1 + \alpha}},
\label{errorlin}
\eeal
and therefore the error tends to zero as $\norm{E}_{V_2}$ tends to zero. This inequality is obtained for an optimal value of~$\lb$, 
\[
\lb_{\text{opt}} =  K_2  \left(\frac{\norm{Q}_{V_2}}{\norm{E}_{V_2}}\right)^{\frac{1}{1+\alpha}}.
\]
\end{thm}
\begin{rem}
Let us give an example of the inequality \eqref{errorlin} and show the interest of this theorem: a particular norm for the error can vanish even if the $L^2$ norm does not. We choose the noise
$$E(x) = \varepsilon^r \cos\Bigl(\frac{x}{\varepsilon}\Bigr), \quad \dEalpha(x) = \varepsilon^{r - \alpha} \cos\Bigl(\frac{x}{\varepsilon} + \alpha \frac{\pi}{2}\Bigr).$$
In this example,  $\partial_{\alpha}E(x)$ may be big in $L^{2}$ norm, even if E is not. With $\gamma = 1 - 2\alpha$, we have formally from \eqref{l2normerrorlin}
\begin{equation*}
\norm{g}_{V_2}\leq K_1 \varepsilon^{\frac{1}{1+\alpha}} + K_2\varepsilon^{r - \alpha},
\label{examplebruit}
\end{equation*}
which does not necessarily vanish with $\varepsilon$ if $r \leq \alpha$ whereas Theorem~\ref{thmnoise1lin} yields
\begin{equation*}
\norm{ \widehat \rho_{\psi}}_{\dot H^{\gamma/2}(\R_{x})}^2 \leq K_3\varepsilon^{r + \frac{r\alpha}{1+\alpha}},
\label{examplebruit2}
\end{equation*}
which vanishes inconditionally with $\varepsilon$.
\end{rem}
\begin{rem} 
Compared to our previous results, we are eventually capable of giving an expression of a bounded optimal $\lb$. With the noise here indeed, the theorems of Section~\ref{subsec:nudgingscl} are no longer valid ($\lb$ cannot be taken as big as we want) and we need to carefully choose the nudging coefficient in order not to amplify the noise. See Section~\ref{secnumresults} for the numerical illustration of this theorem on Burgers' equation.
\end{rem}
\begin{proof}[Proof of Theorem~\ref{thmnoise1lin}]
Performing space Fourier transform on \eqref{eqnoise30lin} yields
\[
\xi\cdot ik \widehat g + \lb \widehat g = -\widehat Q + \lb k^{\alpha}\widehat E. 
\]
Then we can deduce an expression of the space Fourier transform of $g$:
\[
\widehat g = \frac{-\widehat Q}{\lb + i \xi\cdot k} + \lb k^{\alpha} \frac{\widehat E}{\lb + i \xi\cdot k}. 
\]
Young's inequality yields
\[
\megaabs{\widehat \rho_{\psi}(t,k)}^2 \leq 2 \megaabs{\int_{\R_{\xi}}{\frac{-\widehat Q(k,\xi)\psi(\xi)}{\lb + i \xi\cdot k}d\xi}}^2 + 2\lb^2 \abs{k}^{2\abs{\alpha}}\megaabs{\int_{\R_{\xi}}{\frac{-\widehat E(k,\xi)\psi(\xi)}{\lb + i \xi\cdot k}d\xi}}^2. 
\]
Applying Cauchy-Schwarz inequality leads to
\begin{align}
\megaabs{\widehat \rho_{\psi}(t,k)}^2 & \leq  \frac{2}{\lb^2} \megaabs{\int_{\R_{\xi}}{\widehat Q^2(k,\xi)d\xi}}\megaabs{\int_{\R_{\xi}}{\frac{\psi(\xi)^2}{1 + \frac{(\xi\cdot k)^2}{\lb^2}}d\xi}} \nonumber\\
 & \hspace{2cm}+ 2 \abs{k}^{2\abs{\alpha}} \megaabs{\int_{\R_{\xi}}{\widehat E^2(k,\xi)d\xi}} \megaabs{\int_{\R_{\xi}}{\frac{\psi(\xi)^2}{1 + \frac{(\xi\cdot k)^2}{\lb^2}}d\xi}} \nonumber\\
&\hspace{-0.7cm}\leq 2 \norm{\psi}_{L^{\infty}(\R_{\xi})} ^2 C_{\psi} \min\left(1,\frac{\lambda}{\abs{k}}\right)\left( \frac{1}{\lb^2} \int_{\R_{\xi}}{\abs{\widehat Q}^2d\xi}+ \abs{k}^{2\abs{\alpha}} \int_{\R_{\xi}}{\abs{\widehat E(k,\xi)}^2d\xi}\right). \label{errorbound}
\end{align}
The last line uses on the one hand the compact support of $\psi$ with
\[
\int_{\R_{\xi}}{\frac{\psi(\xi)^{2}}{1 + \frac{(\xi\cdot k)^2}{\lb^2}}d\xi} \leq \norm{\psi}_{L^{\infty}}^{2} C_{\psi}, 
\]
but also the usual observation leading to averaging lemmas i.e.
\[
\int_{\R_{\xi}}{\frac{\psi(\xi)^{2}}{1 + \frac{(\xi\cdot k)^2}{\lb^2}}d\xi} \leq \norm{\psi}_{L^{\infty}}^{2} C_{\psi} \frac{\lambda}{\abs{k}}, 
\]
where $C_{\psi}$ is  constant related to the support of $\psi$. To prove it, we choose an orthonormal basis for the $\xi$ space, with first vector $\xi_{1} = \frac{k}{\abs{k}}$. The integration in the orthogonal directions of $\xi_{1}$ yields the constant $C_{\psi}$. In the direction of $k$, the inequality reduces to the computation of
\[
\int_{-\infty}^{\infty}{\frac{1}{1 + (\xi_{1}\frac{k}{\abs{k}}\cdot \frac{k}{\lb})^2}d\xi_{1}} = \int_{-\infty}^{\infty}{\frac{1}{1 + \xi_{1}^{2} \frac{\abs{k}^{2}}{\lb^{2}}}d\xi_{1}} = \frac{\lb}{\abs{k}}\int_{-\infty}^{\infty}{\frac{1}{1 + \eta^{2}}d\eta} = \pi \frac{\lb}{\abs{k}}. 
\]
Given the value of $\abs{k}$, either $1$ or $\frac{\lb}{\abs{k}}$ should be chosen in the upper bound \eqref{errorbound} and that is why we introduce $k_{0} \in ]0,+\infty[$ in order to separate the integration domain. Now let us choose $\gamma$ such that $0 < \gamma < 1 - 2\abs{\alpha}$. Then the upper bound of $\megaabs{\rho_{\psi}}^2$ in \eqref{errorbound} is clearly integrable over $\mathcal D_{k_0} = \left\{ k\in \R_{k}, s.t. \ \abs{k} > k_{0}\right\}$ and
\[
\int_{\mathcal D_{k_0}}{\abs{k}^{\gamma}\megaabs{\widehat \rho_{\psi}(t,k)}^2 dk} \leq J_1(\lb, k_{0}), 
\]
where
\[
J_1(\lb, k_{0}) = 2  C_{\psi}\norm{\psi}_{L^{\infty}(\R_{\xi})} ^2 \left( \norm{Q}_{V_2} ^2 \frac{1}{\lb} \frac{1}{k_{0}^{1 - \gamma}}
 +\norm{E}_{V_2} ^2 \lb \frac{1}{k_{0}^{1 - 2\abs{\alpha} - \gamma}}
\right). 
\]
The optimization of this function with respect to $\lb$ leads to
\be
\lb_{\text{opt}}(k_{0}) = \frac{\norm{Q}_{V_2}}{\norm{E}_{V_2}} \frac{1}{k_{0}^{\abs{\alpha}}},
\label{lbopt1lin}
\ee
and
\begin{equation*}
J_1(\lb_{\text{opt}}(k_{0}), k_{0}) = 4 C_{\psi}  \norm{\psi}_{L^{\infty}(\R_{\xi})} ^2 \left( \norm{Q}_{V_2}\norm{E}_{V_2} \frac{1}{k_{0}^{1 - \abs{\alpha}-\gamma}}\right).
\label{jlambdaoptlin}
\end{equation*}
Besides, the integration over $\R^d - \mathcal D_{k_0}$ combined with the value of $\lb$ in \eqref{lbopt1lin} leads
\[
\int_{\abs{k} < k_0}{\abs{k}^{\gamma}\megaabs{\widehat \rho_{\psi}(t,k)}^2 dk} \leq J_2(\lb, k_{0}), 
\]
where
\[
J_2(\lb_{\text{opt}}(k_0), k_{0}) = 4  C_{\psi}\norm{\psi}_{L^{\infty}(\R_{\xi})} ^2 \norm{E}_{V_2} ^2 k_{0}^{2\alpha + \gamma}. 
\]
Eventually, 
\[
\norm{\widehat \rho_{\psi}}_{\dot H^{\gamma/2}(\R_{k}^{d})}^2 \leq J_1(\lb_{\text{opt}}(k_0), k_{0}) + J_2(\lb_{\text{opt}}(k_0), k_{0}) = J(\lb_{\text{opt}}(k_{0}), k_0). 
\]
The previous bound clearly admits a minimum in $k_{0}$ since $J(\lb_{\text{opt}}, \cdot)$ is contintuous on $\left] 0, +\infty \right[$ and 
$$J(\lb_{\text{opt}}(k_{0}), k_0) \xrightarrow[k_{0} \to +\infty] {}+\infty, \quad J(\lb_{\text{opt}}(k_{0}), k_0) \xrightarrow[k_{0} \to +\infty] {} +\infty.$$
The optimisation of $J$ with respect to $k_{0}$ leads to
\[
\argmin_{k_{0}} J = \left( \frac{\norm{Q}_{V_2}}{\norm{E}_{V_2}}\right)^{ \frac{1}{1+\alpha}} \left( \frac{1 - \gamma - \alpha}{2\alpha + \gamma} \right)^{ \frac{1}{1+\alpha}},
\]
and
\[
\min J = 4  C_{\psi}\norm{\psi}_{L^{\infty}(\R_{\xi})}^2 \left(\!\!
 \Bigl( \!\!\frac{2\alpha + \gamma}{1 - \gamma - \alpha} \Bigr)^{ \frac{1 - \alpha - \gamma}{1+\alpha}}
\!\!\!\!\!\!+
 \Bigl( \frac{1 - \gamma - \alpha}{2\alpha + \gamma} \Bigr)^{ \frac{2\alpha + \gamma}{1+\alpha}}
 \right) 
 \norm{E}_{V_2}^{\frac{2 - \gamma}{1 + \alpha}}  \norm{Q}_{V_2}^{\frac{2\alpha + \gamma}{1 + \alpha}}.
\]
Moreover, taking into account the optimal value of $k_{0}$, the optimal $\lb$ is
\[
\lb_{\text{opt}} =  \Bigl( \frac{1 - \gamma - \alpha}{2\alpha + \gamma} \Bigr)^{ \frac{\alpha}{1+\alpha}}  \left(\frac{\norm{Q}_{V_2}}{\norm{E}_{V_2}}\right)^{ \frac{1}{1+\alpha}}. 
\]
\end{proof}

\begin{rem}
The same study is extended to the evolution case. It uses the Fourier transform in $x$, $\xi$ but not in $t$ (see \cite{bouchut1999averaging}). Therefore, it suits better functions defined on $\left[ 0,T\right]\times \R_{x}^d\times\R_{\xi}$. However, the computation is tedious and does not bring more insights about the noisy case. We do not detail it here.
\end{rem}
 

\subsubsection{Nonlinear coefficient}

Let us now consider a nonlinear coefficient $a \in L^{\infty}_{loc}(\R_{\xi}, \R^{d})$, $\xi \in \R$, which corresponds to the kinetic formulation of a general scalar conservation law of the form \eqref{conservlaw} with $p = 1$. In this situation, a non-degeneracy assumption is required to recover the regularity of averages in $\xi$. More precisely, for any direction $\sigma \in \R^{d}$, such that $\abs{\sigma} = 1$, $a(\xi)\cdot \sigma$ has to be non constant to enable the averaging procedure to give regularity in $x$ in the direction $\sigma$.

\label{subsecnoise2}
Let us study the steady state situation. As stated previously, the noise is an $\alpha^{th}$ space derivative of an $L^{2}$ function. We consider the solution of
\beal
a(\xi)\cdot \nabla_{x}f(x,\xi) + \lb f(x,\xi) = \lb M(x,\xi) + \lb \frac{\partial^{\alpha}E}{\partial x^{\alpha}},
\label{eqnoise10}
\eeal
with
\be
a(\xi)\cdot\nabla_{x}M(x,\xi) = Q (x,\xi) \in  L^{2}(\R_{x}^{N}\times\R_{\xi}),
\label{eqnoise20}
\ee
where $a$ is a non-degenerate flux. The error equation writes now:
\be
a(\xi)\cdot \nabla_{x}g(x,\xi) + \lb g(x,\xi) =  \lb \frac{\partial^{\alpha}E}{\partial x^{\alpha}} - Q (x,\xi),
\label{eqnoise30}
\ee
with an explicit solution (Lemma~\ref{lemexistence})
\[
g(x, \xi) = -\frac{1}{\lb}\int_{0}^{+\infty}{\left(Q(x - \frac{a(\xi)}{\lb}s,\xi) + \lb \frac{\partial^{\alpha}E}{\partial x^{\alpha}}(x - \frac{a(\xi)}{\lb}s,\xi)\right)e^{-\lb s}ds}, 
\]
leading to the upper bound of the error in the $L^{2}$ norm
\begin{equation*}
\norm{g}_{V_2} \leq \frac{\norm{Q}_{V_2}}{\lb} + \norm{\dEalpha}_{V_2}.
\label{l2normerror}
\end{equation*}
Let us denote by $\widehat g (k,\xi)$, $\widehat E$, $\widehat Q$ and $\widehat \rho_{\psi}$ the space Fourier transforms of $g$, $E$, $Q$ and $\rho_{\psi}$.
\begin{thm}
\label{thmnoise1}
Let $M(x,\xi) \in V_{2}$ be solution of \eqref{eqnoise20} and $f(x,\xi) \in V_{2}$ be solution of the noisy data assimilation problem \eqref{eqnoise10}. Then, there exists $\beta$, $0< \beta < 1$, $\gamma$, $0 < \gamma < \beta - 2\abs{\alpha}$ and constants $K_1(\alpha, \beta, \gamma), K_2(\alpha, \beta, \gamma)$ such that

\[
\norm{ \widehat \rho_{\psi}}_{\dot H^{\gamma/2}(\R_{x})}^2 \leq K_1 \norm{E}_{V_2}^{\frac{2 - \gamma}{1 + \alpha}}  \norm{Q}_{V_2}^{\frac{2\alpha + \gamma}{1 + \alpha}}, 
\]
which is obtained for an optimal value of $\lb$

\[
\lb_{\text{opt}} =  K_2  \left(\frac{\norm{Q}_{V_2}}{\norm{E}_{V_2}}\right)^{\frac{1}{1+\alpha}}. 
\]
\end{thm}
\begin{proof}[Proof of Theorem~\ref{thmnoise1}]
By means of Fourier transform on \eqref{eqnoise30} we obtain
\[
\widehat g = \frac{-\widehat Q}{\lb + i a(\xi)\cdot k} + \lb k^{\alpha} \frac{\widehat E}{\lb + i a(\xi)\cdot k}. 
\]
Young's inequality followed by a Cauchy-Schwarz inequiality as in the linear case yields
\begin{multline*}
\megaabs{\widehat \rho_{\psi}(t,k)}^2  \leq \frac{2}{\lb^2} \megaabs{\int_{\R_{\xi}}{\widehat Q^2(k,\xi)d\xi}}  \megaabs{\int_{\R_{\xi}}{\frac{\psi(\xi)^2}{1 + \frac{(a(\xi)\cdot k)^2}{\lb^2}}d\xi}}\\
+ 2 k^{2\alpha} \megaabs{\int_{\R_{\xi}}{\widehat E^2(k,\xi)d\xi}} \megaabs{\int_{\R_{\xi}}{\frac{\psi(\xi)^2}{1 + \frac{(a(\xi)\cdot k)^2}{\lb^2}}d\xi}}. 
\end{multline*}
Under the non-degeneracy condition for the flux $a(\xi)$, i.e. the existence of $0 < \beta < 1$ s.t.
\[
\int_{\R_{\xi}}{\frac{\psi(\xi)^{2}}{1 + \frac{(a(\xi)\cdot k)^2}{\lb^2}}d\xi} \leq \norm{\psi}_{L^{\infty}(\R_{\xi})}^{2} C_{\psi} \frac{\lambda^{\beta}}{\abs{k}^{\beta}}, 
\]
we can state that

\[
\megaabs{\widehat \rho_{\psi}(t,k)}^2 \leq 2 C_{\psi}\norm{\psi}_{L^{\!\infty}(\!\R_{\xi})} ^2 \min\Bigl( 1, \frac{\lb^{\beta}}{\abs{k}^{\beta}}\Bigr)\!\! \left( \frac{1}{\lb^2}\!\!\!\int_{\R_{\xi}}{\!\!\!\abs{\widehat Q}^2d\xi} + k^{2\alpha}\!\! \int_{\R_{\xi}}{\!\!\!\abs{\widehat E(k,\xi)}^2d\xi} \right)\!\!\!. 
\]
Now let us choose $\gamma$ such that $0 < \gamma < \beta - 2\abs{\alpha}$. From this moment on, the proof does not differ from the one of Theorem~\ref{thmnoise1lin}. it consist in minimizing the upper bound first with respect to $\lb$ and then with respect to a limiting frequency that we denoted $k_0$. We will not detail it furthermore as a complete study is available in \cite{boulangerthese}.
\end{proof}
\begin{rem}
 The proofs of Theorem~\ref{thmnoise1lin} and Theorem~\ref{thmnoise1} are based on finding an optimal value for $\lb$. The existence of an optimum is illustrated on Burgers' equation in Section~\ref{subsec:svtnumerics}, Figure~\ref{lambdaopt}.
\end{rem}

\begin{rem}
As in the linear case, the same study can be performed in the evolution case, using the Fourier transform in $x$, $\xi$ but not in $t$. But heavy computations are needed and we do not detail them in the article.
\end{rem}

\subsection{Time sampling operator in the scalar case}
In the very theoretical case of complete and exact observations, as studied in Section~\ref{subsec:nudgingscl}, the conclusion is that the scalar nudging term can be chosen as big as possible and convergence is ensured. When adding noise, an optimal value for $\lb$ minimizes the covergence error, as explained in Section~\ref{secnoise}. Now we address another difficulty, namely that access to the observations is not possible at all times but only at time $t_k$. Moreover it is not an exact observation but a mean of it over a small period of time. For that purpose, we define a mollifier $\varphi$ i.e. a nonnegative smooth and compactly supported function with integral one. From this mollifier we build a regularizing sequence $$\left\{ \varphi_{\sigma} =  \frac{1}{\sigma}\varphi(\frac{s}{\sigma}), \ \sigma \in \R^+ \right\}.$$
Then, considering that data come from a hyperbolic conservation law of type \eqref{conservlaw} admitting a kinetic formulation for the density $M(t,x,\xi)$:
\beal
\partial_{t}M(t,x,\xi) + a(\xi)\cdot\nabla_{x}M(t,x,\xi) = Q(t,x,\xi),
\label{eqobstime}
\eeal
for some collision term Q, we work with the following equation of the estimator:
\beal
\partial_{t}f(t,x,\xi) + a(\xi)\cdot\nabla_{x}f(t,x,\xi) +  \lambda \sum_{k\in I}{\varphi_{\sigma}(t-t_{k})f(t_{k},x,\xi)}= \lambda \sum_{k\in I}{\varphi_{\sigma}(t-t_{k})M(t_{k},x,\xi)},
\label{eqobstime2}
\eeal
where $\left\{ t_{k},\, k \in I\right\}$ are the observation times.
\begin{thm}
Let us suppose that $M$ is solution of \eqref{eqobstime} and $f$ is solution of
\eqref{eqobstime2} such that $f - M$ is of time bounded variations. Provided that $\Card(k, t_k \leq T - \sigma) \neq 0$, then,
\begin{multline*}
\norm{f(T) - M(T)}_{V_1}\leq \norm{f_0 -  M_0}_{V_1}e^{-\lb \Card(k, t_k \leq T - \sigma)} 
\\+ \sigma \norm{\partial_t (f - M)}_{L_t^{\infty}(M_x^1(\R_x), L^{1}(\R_{\xi}))} 
+ \frac{1}{\lb}\sup_{0 < t \leq T}\norm{Q(t)}_{V_1}. 
\end{multline*}
\label{thmsamplingtime}
\end{thm}
\begin{proof}[Proof of Theorem~\ref{thmsamplingtime}]
\noindent We perform the difference between \eqref{eqobstime2} and \eqref{eqobstime} and begin
with
\begin{multline*}
\partial_{t}(f - M)(t,x,\xi) + a(\xi)\cdot \nabla_{x}(f - M)(t,x,\xi) \\+ \lambda \sum_{k\in I}{\varphi_{\sigma}(t-t_{k})(f - M)(t_{k},x,\xi)} = - Q(t,x,\xi).
\end{multline*}
\noindent Multiplying by $\sgn(f - M)(t,x,\xi)$ and integrating in $(x,\xi)$ leads to
\begin{multline*}
\f d {dt} \norm{(f - M)(t)}_{V_1} \\ 
+ \lambda
\int_{\R_{x}}{\int_{\R_{\xi}}{\sum_{k\in I}{\varphi_{\sigma}(t-t_{k})(f - M)(t_{k},x,\xi)\sgn(f - M)(t,x,\xi)}d\xi}dx} 
\leq  \norm{Q(t)}_{V_1}.
\end{multline*}
Using Fubini theorem to invert the integral and sum operators, the previous
equation also writes
\begin{multline*}
\f d {dt} \norm{(f - M)(t)}_{V_1} + \lambda \sum_{k\in
I}{\varphi_{\sigma}(t-t_{k})\norm{(f - M)(t)}_{V_1}}
 \\  \leq \lambda \int_{\R_{x}}{\int_{\R_{\xi}}{\sum_{k\in I}{\varphi_{\sigma}(t-t_{k})\abs{(f - M)(t,x,\xi)-(f - M)(t_{k},x,\xi)}}d\xi}dx} + \norm{Q(t)}_{V_1}.
\end{multline*}
Since $f - M$ is of time bounded variation, we get
\begin{multline*}
\f d {dt} \norm{(f - M)(t)}_{V_1}  + \lambda \sum_{k\in
I}{\varphi_{\sigma}(t-t_{k})\norm{(f - M)(t)}_{V_1}} \\ 
\leq \lambda \int_{\R_{x}}{\int_{\R_{\xi}}{\sum_{k\in
I}{\varphi_{\sigma}(t-t_{k})\megaabs{\int_{t}^{t_k}{\partial_{\tau} ( f - M)
(\tau,x,\xi)d\tau}}d\xi}dx}} +  \norm{Q(t)}_{V_1}.
\end{multline*}
Once again, Fubini theorem yields
\begin{multline*}
\f d {dt} \norm{(f - M)(t)}_{V_1} + \lambda \sum_{k\in
I}{\varphi_{\sigma}(t-t_{k})\norm{(f - M)(t)}_{V_1}} \\
\leq \lambda \sum_{k\in I}{\varphi_{\sigma}(t-t_k) \sgn(t -t_{k}) \int_{t_{k}}^{t}{\int_{\R_x}{\int_{\R_{\xi}}{\abs{\partial_{\tau}(f - M)}d\xi}dx}d\tau}} + \norm{Q(t)}_{V_1}.
\end{multline*}
Recalling that $\partial_t (f - M) \in L_t^{\infty}(M_x^1(\R_x^d), L^{1}(\R_{\xi}))$, we can write
\begin{multline*}
\f d {dt} \norm{(f - M)(t)}_{V_1}+ \lambda \sum_{k\in I}{\varphi_{\sigma}(t-t_{k})\norm{(f - M)(t)}_{V_1}} \\ 
\leq \lambda \sum_{k\in I}{\varphi_{\sigma}(t-t_{k}) \abs{t - t_k} \norm{\partial_t (f - M)}_{L_t^{\infty}(M_x^1(\R_x^d))}} + \norm{Q(t)}_{V_1}.
\end{multline*}
Moreover, because of the compact domain of each $\varphi_{\sigma}$ we can state
\begin{multline*}
\f d {dt} \norm{(f - M)(t)}_{V_1} + \lambda \sum_{k\in
I}{\varphi_{\sigma}(t-t_{k})\norm{(f - M)(t)}_{V_1}} \\ 
\leq \lambda \sigma \norm{\partial_t (f - M)}_{L_t^{\infty}(M_x^1(\R_x))}\sum_{k\in
I}{\varphi_{\sigma}(t-t_{k})} + \norm{Q(t)}_{V_1}.
\end{multline*}
And finally using Gronwall's lemma we obtain
\begin{multline*}
\norm{f(T) -  M(T)}_{L^1(\R_x \times \R_{\xi})} \\
\hspace{-3.5cm} \leq \norm{f_0 - M_0}_{L^1(\R_x \times \R_{\xi})}e^{-\lb \int_0^T{\sum_{k \in I}{\varphi_{\sigma}(t - t_k)}dt}} \\
\hspace{1.1cm} + \sigma \norm{\partial_t (f - M)}_{L_t^{\infty}(M_x^1(\R_x), L^{1}(\R_{\xi}))} \left( 1 - e^{-\lb \int_0^T{\sum_{k \in I}{\varphi_{\sigma}(t - t_k)dt} }}\right)\\
+ \sup_{0 < t \leq T}\norm{Q(t)}_{V_1} \frac{1 - e^{-\lb \int_0^T{\sum_{k \in I}{\varphi_{\sigma}(t - t_k)dt} }}}{\lb}.
\end{multline*}
The conclusion follows.
\end{proof}
\begin{rem}
From Theorem~\ref{thmsamplingtime}, we have the following remarks. The error in the initial condition is forgotten with time, as expected from hyperbolic conservation laws. The effect of the collision term is filtered by the nudging coefficient $\lb$. However, the error in time on the observations introduces a bias of order $\sigma$ in the upper bound which cannot be ignored nor counter-balanced. But it is worth noticing that $\lb$ does not amplify this error which is very important point.
\end{rem}

\subsection{Incomplete observations in space}
In this section, a subdomain is defined, where assimilation towards data occur. We deal with a simplified 1D problem with boundary conditions and assuming that solutions are smooth. The first idea is to find the minimal time required for the information to spread over the whole domain and have an effect on the variables. In control theory, this is the observability problem, dealt with in Section~\ref{observability}. The second idea is to prove the stability of the problem by studying the error decay, which is analyzed in Section~\ref{stability}.

\subsubsection{Observability}
\label{observability}

Let us study the periodic 1D problem over the domain $\left[ 0,1 \right]$

\bealn
  &\partial_t f(t,x,\xi) + a(\xi)\cdot \nabla_x f(t,x,\xi) = \lambda(x)(M(t,x,\xi) - f(t,x,\xi)), \\
  &f(0,x,\xi) = f_0(x,\xi); \quad f(t,0,\xi) = f(t,1,\xi).
\label{parobs}
\eealn

where the observations are supposed to come from the model
\bealn
  &\partial_t M(t,x,\xi) + a(\xi)\cdot \nabla_x M(t,x,\xi) = Q(t,x,\xi),\\
  &M(0,x,\xi) = M_0(x,\xi); \quad M(t,0,\xi) = M(t,1,\xi).
\label{parobsex}
\eealn
We are interested in the case $\lb(x) = \lb \1_{\left[ a,b \right]}(x)$, $\lb \in \R^+, \ 0<a<b<1$. We consider $Q = 0$, $\xi$ is supposed to be fixed. This assumption is not meaningless. It represents indeed the situation where no collision occur at the kinetic level, which translates at the macroscopic level into a situation where no shock occur.

The first concern is the observability of the situation at time $T >0$, where observability is defined as follows in our case
\begin{defn}[Observability condition]
The Cauchy system \eqref{parobs} is observable if and only if there exists $C \in \R_+^*$ such that, for all $f$ solution of \eqref{parobs} with $\lambda = 0$, 
$$C \int_0^T \norm{\lambda(x) u(x,t) }_{L^2(\left[ 0,1 \right])}^2 \geq  \norm{ u(0) }_{L^2(\left[ 0,1 \right]}^2.$$
\end{defn}
This formally means that the information captured in the time lapse $\left[ 0,T \right]$ in the subdomain $\left[ a,b \right]$ allows to reconstruct any initial condition. At fixed $\xi$, $\left( s, x+a(\xi)s\right)$ is the characteristic curve starting in $x$ of the pure transport problem ($\lambda = 0$). The quantity $\int_0^{t}{\1_{\left[ a,b \right]}(x + a(\xi)s)ds}$ represents the time that the particle starting at point $x$ spends in interval $\left[ a,b \right]$ between $0$ and $t$. Let us call
$$ X_{t,\xi} (x) =  \int_0^{t}{\1_{\left[ a,b \right]}(x + a(\xi)s)ds}. $$
We prove the proposition:
\begin{prop}
$X_{t,\xi}^{\text{inf}} = \displaystyle \min_{x} X_{t,\xi} > 0 \Rightarrow \eqref{parobs}$ is observable.
\label{propobservability}
\end{prop}
 
\begin{proof}[Proof of Proposition~\ref{propobservability}]
Let us show first that the condition $X_{t,\xi}^{\text{inf}} > 0$ (every charactetistics spends time in the observed domain) is equivalent to $T > \frac{1 - (b-a)}{\abs{a(\xi)}}$. This shows that $a(\xi)$ should necessarily be non-zero, reflecting the fact that particles with speed 0 are not observable. On Figure~\ref{Charac} we represent the characteristics of constant postitive speed $c$ in the periodic case on domain $\left[ 0,1 \right]$. We see that the last characteristics to pass through the domain $\left[ a,b \right]$ is the blue one, starting right after point $b$. This happens if $T > \frac{1 - (b-a)}{c}$.

\begin{figure}[htbp]
\begin{center}
\begin{tikzpicture}[scale=0.65]
\draw[->] (-0.5,0) -- (7,0);
\draw (7,0) node[below] {$x$};
\draw [->] (0,-0.5) -- (0,7);
\draw (0,7) node[left] {$t$};
\draw (0,-0.2) node[left] {$0$};
\draw[black,dashed] (6,0) -- (6,7);
\draw[line width=1,red] (4,0) -- (4,7);
\draw (4,-0.1) node[below] {$a$};
\draw[line width=1,red] (5,0) -- (5,7);
\draw (5,0) node[below] {$b$};
\draw (6,0) node[below] {$1$};
\draw[white,fill=red!30,opacity = 0.5] (4,0) rectangle (5,7);
\draw [domain=0:6] plot(\x,\x);
\draw [blue,domain=5:6] plot(\x,\x-5);
\draw [domain=0:6, densely dotted] plot(\x,1);
\draw [blue,domain=0:6] plot(\x,\x+1);
\draw [domain=0:6, densely dotted] plot(\x,3);
\draw [domain=3:6] plot(\x,\x-3);
\draw [domain=0:4, dashed] plot(\x,5);
\draw (0,5) node[left] {$T = \frac{1 - (b - a)}{c}$};
\draw [domain=0:6, dashed] plot(\x,6);
\draw (0,6) node[left] {$ \frac{1}{c}$};
\draw [domain=0:4] plot(\x,\x+3);
\draw (4,5) node {$\bullet$};
\end{tikzpicture}
\caption[Characteristics of a 1D periodic problem]{Characteristics of the problem. This figure shows the equivalence between both conditions $X_{t,\xi}^{\text{inf}} > 0$ and $T > \frac{1 - (b-a)}{\abs{a(\xi)}}$.}
\label{Charac}
\end{center}
\end{figure}
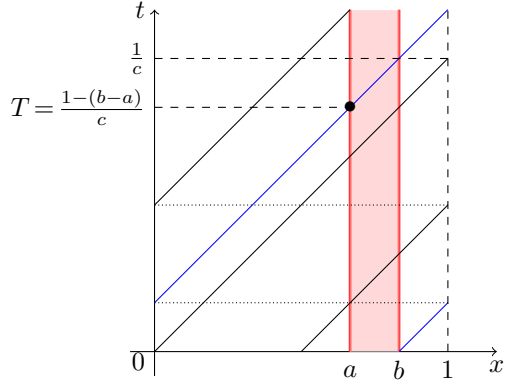
\end{proof}

\subsubsection{Error decay}
\label{stability}

Let us denote $g(t,x,\xi)$ the error between the observations and the estimator:
$$g(t,x,\xi) = f(t,x,\xi) - M(t,x,\xi).$$
\begin{prop}
Let $f(t,x,\xi)$ satisfy \eqref{parobs} and $M(t,x,\xi)$ satisfy \eqref{parobsex} where $Q(t,x,\xi)~=~0$. If the data assimilation problem \eqref{parobs} is observable, exponential stability in time is ensured as soon as $\lambda \geq \frac{t}{X_t^{\text{inf}}}$, where $X_t^{\text{inf}}$ is defined in Section~\ref{observability} i.e.
\begin{equation*}
\norm{g(t)}_{V_2} \leq \norm{g_0}_{V_2}e^{-\lambda X_t^{\text{inf}}}.
\end{equation*}
\label{propstability}
\end{prop}
\begin{proof}[Proof of Proposition~\ref{propstability}]
Subtracting \eqref{parobs} and \eqref{parobsex} leads to
\bealn
&  \partial_t g(t,x,\xi) + a(\xi)\cdot \nabla_x g(t,x,\xi) = -\lambda(x)g(t,x,\xi), \\
 & g(0,x,\xi) = f_0(x,\xi) - M_0(x,\xi) = g_0(x,\xi); \quad g(t,0,\xi) = g(t,1,\xi).
\label{eqerror}
\eealn
\eqref{eqerror} is equivalent to
\bealn
&\frac{d}{dt}g(t, x + a(\xi)t,\xi) = -\lambda(x + a(\xi)t)g(t, x + a(\xi)t,\xi),\\
 &g(0,x,\xi) = g_0(x,\xi); \quad  g(t,0,\xi) = g(t,1,\xi).
\label{eqerror2}
\eealn
\eqref{eqerror2} integrates into
\begin{align*}
 g(t,x+a(\xi)t,\xi) &= g_0(x,\xi)\text{exp}\Bigl({-\int_0^{t}{\lambda(x + a(\xi)s)}ds}\Bigr), \nonumber \\
 &= g_0(x,\xi)\text{exp}\Bigl({-\lambda\int_0^{t}{\1_{\left[ a,b \right]}(x + a(\xi)s)ds}}\Bigr).
\label{eqerror3}
\end{align*}
Then, if $X_{t,\xi}$ has a positive lower bound $X_{t,\xi}^{\text{inf}}$ (if the system is observable), 
\begin{equation*}
 \norm{g(t)} \leq \norm{g_0}e^{-\lambda X_t^{\text{inf}}}. 
\end{equation*}
Therefore, an exponential decrease of the error is possible if $\lambda \geq \frac{t}{X_t^{\text{inf}}}$.
\end{proof}
Then it is possible to add the collision term in a second step. The analytical resolution of the error writes
\begin{multline*}
 g(t,x+a(\xi)t,\xi) = g_0(x,\xi) \text{exp}\Bigl({-\int_0^{t}{\lambda(x + a(\xi)s)}ds} \Bigr) \\+ \int_0^t{Q(s, x+a(\xi)s,\xi)\text{exp}\Bigl({-\int_s^t{\lambda(x+a(\xi)r)dr}}ds}\Bigr). 
\end{multline*}
However, it is not so clear how to treat the integral with the collision term, since we cannot extract the coefficient $\lb$ as easily as in the situation when it was constant. It raises the question to know if the coefficient can filter collisions everywhere, like in the case of complete observations at the beginning of this paper, even if it is spatially constrained. This problem remains an open question and we leave it with no specific answer for now.

\subsection{Numerical results}
\label{secnumresults}
Numerical experiments of the kinetic nudging are performed first on a scalar conservation law, Burgers' equation. Further numerical results about the shallow water equations are presented in Section~\ref{subsec:svtnumerics}. 

\subsubsection{Discretization}

We consider that observations satisfy the Burgers equation with homogeneous Dirichlet boundary conditions on the domain $[0,1]$
\begin{equation*}
   \left\{
       \begin{aligned}
&\partial_t  u + \partial_x \frac{ u^2}{2} = 0, \\
& u(0,x) =  u_0(x); \quad u(t,0) = 0, \ u(t,1) = 0.
\end{aligned}
\right. 
\end{equation*}
This problem with all its entropies is well represented by the kinetic model
\begin{equation*}
   \left\{
       \begin{aligned}
&\partial_{t} \Chi(\xi,  u(t,x)) + \xi \partial_{x} \Chi(\xi,  u(t,x)) = \partial_{\xi}m(t,x,\xi),\\
&\Chi(\xi,  u(0,x)) =  \Chi(\xi,  u_{0}(x)); \quad \Chi(\xi,  u(t,0)) = \Chi(\xi,  u(t,1)) = 0.
\end{aligned}
\right.
\end{equation*}
Therefore, data assimilation at the kinetic level writes
\begin{equation}
   \left\{
       \begin{aligned}
&\partial_{t} f(t,x,\xi) + \xi \partial_{x} f(t,x,\xi) =\lb (\Chi(\xi,  u(t,x)) - f(t,x,\xi)),\\
&f(0,x,\xi) =  f_{0}(x); \quad f(t,0,\xi) = f(t,1,\xi) = 0.
\end{aligned}
\right.
\label{burgerskin}
\end{equation}
To solve this numerically, we develop a finite volume framework. The domain $[0 ,1]$ is divided into N cells centered on $x_{i},\  i \in 0..N$. Then, we integrate \eqref{burgerskin} over $[x_{i-1/2}, x_{i-1/2}] \times [t^n, t^{n+1}]$. An explicit Euler scheme leads to
\begin{align*}
f_i^{n+1} =  f_i^{n} - \frac{\Delta t}{\Delta x} \xi \left( f_{i+1/2}^n - f_{i-1/2}^n\right) + \lambda \Delta t \left(\Chi(\xi, u_i^n) - f_i^n \right).
\end{align*}
The numerical fluxes are computed using a simple and natural upwind scheme
\begin{equation*}
	f_{i+1/2}^n =
	\begin{cases}
	    f_i^{n}, & \text{ for } \xi \geq 0,  \\
	    f_{i+1}^{n}, & \text{ for }  \xi < 0. 
	\end{cases}
\end{equation*}
This explicit scheme is only valid under the CFL condition
$$\Delta t \leq \left(\lambda + \frac{\sup{\xi}}{\Delta x}\right)^{-1}.$$
An integration in $\xi$ gives the macroscopic scheme for the estimator $\widehat u(t,x)$, which is basically the Osher-Engquist scheme
\begin{multline*}
\widehat u_i^{n+1} = \widehat u_i^{n} - \frac{\Delta t}{\Delta x} \left( \max(\widehat u_i^{n},0) \frac{\widehat u_i^{n}}{2}
+ \min(\widehat u_{i+1}^{n},0)\frac{\widehat u_{i+1}^{n}}{2}\right.\\
 \left.-\max(\widehat u_{i-1}^{n},0)\frac{\widehat u_{i-1}^{n}}{2}
-\min(\widehat u_{i}^{n},0)\frac{\widehat u_{i}^{n}}{2}
\right)
+\lb \Delta t (u_i^{n} - \widehat u_i^{n}).
\end{multline*}

\subsubsection{Results}
\label{resultsburgers}

We want to illustrate our mathematical result about noise stability (Theorem~\ref{thmnoise1lin}) on Burgers' equation. We add a deterministic noise to the data. As explained in the Remark~\ref{remnoise}, in practice the noise is on the macroscopic data. Let us choose $$E(x) = \varepsilon \cos\Bigl(\frac{x}{\varepsilon}\Bigr), \quad \partial_{\alpha}E(x) = \varepsilon^{1 - \alpha}\cos\Bigl(\frac{x}{\varepsilon} + \alpha \frac{\pi}{2}\Bigr).$$
Therefore, $$\norm{E}_{L^2(\left[ 0,1 \right])} = \frac{\varepsilon}{2}\sqrt{2 + \varepsilon \sin\Bigl(\frac{2}{\varepsilon}\Bigr)},\quad \norm{\partial_{\alpha} E}_{L^2(\left[ 0,1 \right])} = \frac{\varepsilon^{1 - \alpha}}{2}\sqrt{2 + \varepsilon \sin\Bigl(\frac{2}{\varepsilon}\Bigr)}.$$
The noise may then be non negligible in $L^2$ norm (as soon as $\alpha > 0$) whereas it is small in $H^{-\alpha}$. Initial conditions are taken as follows:
\begin{equation*} u(0,x) = \begin{cases} 1& \text{ if } \frac{1}{8} \leq x \leq \frac{1}{4}\\
                 0& \text{ otherwise, }
                 \end{cases}, \qquad
   \widehat u(0,x) = \begin{cases} 0.75& \text{ if } \frac{1}{12} \leq x \leq \frac{1}{6},\\
                 0& \text{ otherwise. }
                 \end{cases}
\end{equation*}              
The parameters take the following values: $N=100$, $\lambda=100$, $\alpha = \frac{1}{4}$. The observations are partial in time: 30 observations are carried out between $t\,=\,0$ and $t\,=\,2$. The number of time steps is 320 which makes an observation ratio of $10\%$. The first row of Figure~\ref{6shots} shows the evolution of the observer if no assimilation occurred ($\lb$ = 0). We can see that the error introduced by a wrong initial condition is detrimental. Figure~\ref{6shots} shows the evolution of the data assimilation problem for two different values of $\varepsilon$. In the case $\varepsilon = 0.02$, $\norm{\partial_{\alpha} E}_{L^2(\left[ 0,1 \right])} = 4.0 \,10^{-2}$, whereas $\norm{\partial_{\alpha} E}_{L^2(\left[ 0,1 \right])} = 7.0 \,10^{-3}$ for $\varepsilon = 0.002$. In comparison, the norm of the initial solution is $\norm{u(0,\cdot)}_{L^2(\left[ 0,1 \right])} = 4.0\, 10^{-1}$. It converges rapidly to the global shape of the original problem. As stated previously, the noise introduced on the 
data 
cannot be forgotten, but it is not amplified.
\begin{figure}[htbp]
\begin{center}
\begin{tikzpicture}[scale=0.50]
\begin{axis}[xlabel=x, ylabel=u(x),font=\Large]
\addplot[color=red,mark=.,line width=2pt] table[x = x,y=u]
	 {FigDA/testobs_0001.txt};
	 \addlegendentry{Exact $u$}
	 \addplot[color=blue,mark=.,line width=2pt] table[x = x,y=u]
	 {FigDA/testcontrol_0001.txt};
	 \addlegendentry{ $\widehat u$, $\lambda=0$}
\end{axis}
\end{tikzpicture}	 
\begin{tikzpicture}[scale=0.50]
\begin{axis}[xlabel=x, ylabel=u(x),font=\Large]
\addplot[color=red,mark=.,line width=2pt] table[x = x,y=u]
	 {FigDA/testobs_0020.txt};
	 \addlegendentry{ Exact $u$}
	 \addplot[color=blue,mark=.,line width=2pt] table[x = x,y=u]
	 {FigDA/testcontrol_0020.txt};
	 \addlegendentry{ $\widehat u$, $\lambda=0$}
\end{axis}
\end{tikzpicture}
\begin{tikzpicture}[scale=0.50]
\begin{axis}[xlabel=x, ylabel=u(x),ymax=1.1, font=\Large]
\addplot[color=red,mark=.,line width=2pt] table[x = x,y=u]
	 {FigDA/testobs_0050.txt};
	 \addlegendentry{Exact $u$}
	 \addplot[color=blue,mark=.,line width=2pt] table[x = x,y=u]
	 {FigDA/testcontrol_0050.txt};
	 \addlegendentry{ $\widehat u$, $\lambda=0$}
\end{axis}
\end{tikzpicture}

\vspace{0.5cm} 
\begin{tikzpicture}[scale=0.50]
\begin{axis}[xlabel=x, ylabel=u(x),font=\Large]
\addplot[color=red,mark=.,line width=2pt] table[x = x,y=u]
	 {FigDA/burgersdet002_0001.txt};
	 \addlegendentry{Exact $u$}
	 \addplot[color=blue,mark=.,line width=2pt] table[x = x,y=u]
	 {FigDA/burgersdet2002_0001.txt};
	 \addlegendentry{ $\widehat u$, $\lambda=100$}
\end{axis}
\end{tikzpicture}	 
\begin{tikzpicture}[scale=0.50]
\begin{axis}[xlabel=x, ylabel=u(x),font=\Large]
\addplot[color=red,mark=.,line width=2pt] table[x = x,y=u]
	 {FigDA/burgersdet002_0020.txt};
	 \addlegendentry{ Exact $u$}
	 \addplot[color=blue,mark=.,line width=2pt] table[x = x,y=u]
	 {FigDA/burgersdet2002_0020.txt};
	 \addlegendentry{ $\widehat u$, $\lambda=100$}
\end{axis}
\end{tikzpicture}
\begin{tikzpicture}[scale=0.50]
\begin{axis}[xlabel=x, ylabel=u(x),ymax=1.1,font=\Large]
\addplot[color=red,mark=.,line width=2pt] table[x = x,y=u]
	 {FigDA/burgersdet002_0050.txt};
	 \addlegendentry{Exact $u$}
	 \addplot[color=blue,mark=.,line width=2pt] table[x = x,y=u]
	 {FigDA/burgersdet2002_0050.txt};
	 \addlegendentry{ $\widehat u$, $\lambda=100$}
\end{axis}
\end{tikzpicture}

\vspace{0.5cm} 

\begin{tikzpicture}[scale=0.50]
\begin{axis}[xlabel=x, ylabel=u(x),font=\Large]
\addplot[color=red,mark=.,line width=2pt] table[x = x,y=u]
	 {FigDA/burgersdet0002_0001.txt};
	 \addlegendentry{Exact $u$}
	 \addplot[color=blue,mark=.,line width=2pt] table[x = x,y=u]
	 {FigDA/burgersdet20002_0001.txt};
	 \addlegendentry{ $\widehat u$, $\lambda=100$}
\end{axis}
\end{tikzpicture}	 
\begin{tikzpicture}[scale=0.50]
\begin{axis}[xlabel=x, ylabel=u(x),font=\Large]
\addplot[color=red,mark=.,line width=2pt] table[x = x,y=u]
	 {FigDA/burgersdet0002_0020.txt};
	 \addlegendentry{Exact $u$}
	 \addplot[color=blue,mark=.,line width=2pt] table[x = x,y=u]
	 {FigDA/burgersdet20002_0020.txt};
	 \addlegendentry{ $\widehat u$, $\lambda=100$}
\end{axis}
\end{tikzpicture}
\begin{tikzpicture}[scale=0.50]
\begin{axis}[xlabel=x, ylabel=u(x),ymax=1.1,font=\Large]
\addplot[color=red,mark=.,line width=2pt] table[x = x,y=u]
	 {FigDA/burgersdet0002_0050.txt};
	 \addlegendentry{Exact $u$}
	 \addplot[color=blue,mark=.,line width=2pt] table[x = x,y=u]
	 {FigDA/burgersdet20002_0050.txt};
	 \addlegendentry{ $\widehat u$, $\lambda=100$}
\end{axis}
\end{tikzpicture}
	 \caption[Evolution of the data assimilation problem on Burgers' equation]{Top: Evolution of the data assimilation problem at time t~=~0.01, t~=~0.2 and t~=~0.5 in the case $\varepsilon~=~0.02$. Bottom: Evolution of the data assimilation problem at time t ~=~0.01, t~=~0.2 and t~=~0.5 in the case $\varepsilon~=~0.002$.}
	 \label{6shots}
	 \end{center}
\end{figure}
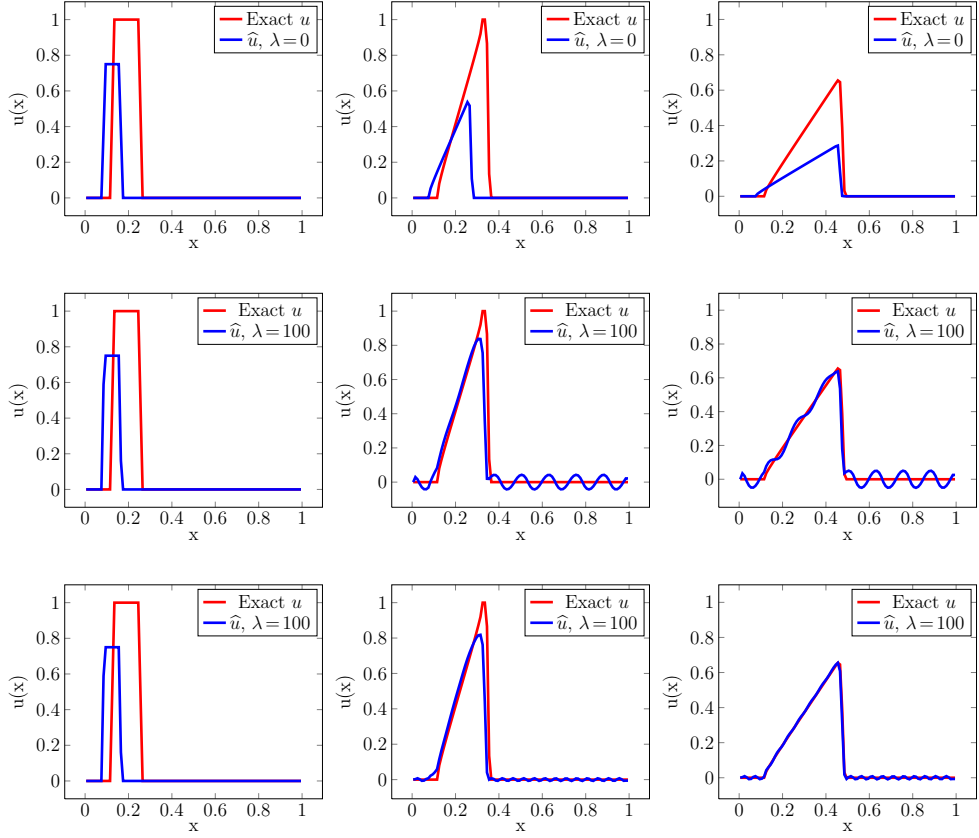

The proofs of the theorems about noise stability are centered around the search for an optimal $\lambda$. Figure~\ref{lambdaopt} shows the existence of this optimal $\lambda$ minimizing the homogeneous Sobolev seminorm. Moreover, it shows that the error tends to zero as $\varepsilon$ tends to zero.
\begin{figure}[htbp]
\begin{center}
\begin{tikzpicture}[scale=0.9]
\begin{axis}[
xlabel=$\lambda$,
ylabel=$H^{1/8}$ norm of the error,
legend style={font=\fontsize{4}{5}\selectfont}]
]
	 \addplot[color=blue,mark=x] table[x = gain,y=error]
	 {FigDA/normsobeps5m2bis.txt};
	 \addlegendentry{$\varepsilon = 0.05$}
	 	 \addplot[color=cyan,mark=x] table[x = gain,y=error]
	 {FigDA/normsobeps4m2bis.txt};
	 \addlegendentry{$\varepsilon = 0.04$}
	 	 \addplot[color=mygreen,mark=x] table[x = gain,y=error]
	 {FigDA/normsobeps3m2bis.txt};
	 \addlegendentry{$\varepsilon = 0.03$}
	 	 \addplot[color=myorange,mark=x] table[x = gain,y=error]
	 {FigDA/normsobeps2m2bis.txt};
	 \addlegendentry{$\varepsilon = 0.02$}
	 \addplot[color=magenta,mark=x] table[x = gain,y=error]
	 {FigDA/normsobeps1m2bis.txt};
	 \addlegendentry{$\varepsilon = 0.015$}
	 \addplot[color=red,mark=x] table[x = gain,y=error]
	 {FigDA/normsobeps5m3bis.txt};
	 \addlegendentry{$\varepsilon = 0.005$}
	\addplot[color=black,mark=x] table[x = gain,y=error]
	 {FigDA/normsobeps1m4bis.txt};
	 \addlegendentry{$\varepsilon = 0.0001$}	
\end{axis}
\end{tikzpicture}
\caption[Optimal $\lb$ detection on Burgers' equation]{$H^{1/8}$ error between the observations and the model as a function of $\lambda$. We see on the one hand the existence of an optimal $\lambda$ for any value of $\varepsilon$. Notice also the decrease of the error with~$\varepsilon$. }
\label{lambdaopt}
\end{center}
\end{figure}
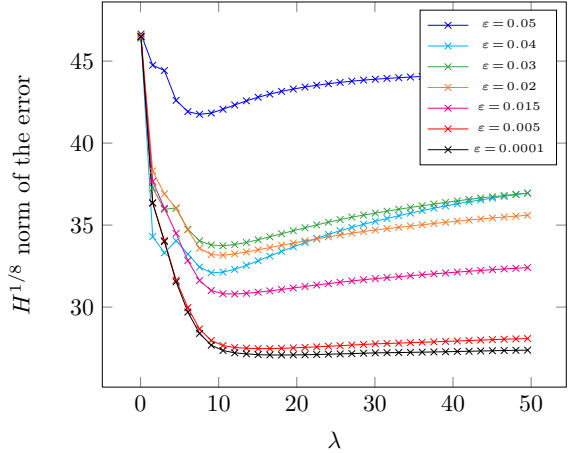

In the end, the observer and the numerical scheme for the data assimilation is similar to what could have been done on the macroscopic conservation law directly. However, this framework paves the way for an interesting theoretical framework, making the most of what has already been around kinetic equations. But most of all, it is its extension to the case of hyperbolic systems in Section~\ref{sec:saintvenant} that reveals its efficiency.

%% file: saintvenant.tex

\section{Kinetic Luenberger observer on the Saint-Venant system}
\label{sec:saintvenant}
Systems in several space dimensions usually hardly admit a kinetic formulation. They satisfy indeed a single entropy inequality instead of a whole family. And this is not enough to deduce the sign of the kinetic entropy defect measure in \eqref{kinprelim}. Therefore, we introduce a weaker concept: the kinetic representation.

\subsection{Kinetic representation of the Saint-Venant system}
\label{subsec:kindesc}
The kinetic representation only uses one entropy and it is therefore much less demanding than the kinetic formulation. It enables to represent a hyperbolic system by integration of the underlying kinetic equation. Its main interest lies in the design of finite volume solvers. We will show that it is also very useful to derive observers. Since the kinetic representation does not admit a general writing, we choose to study the Saint-Venant system.

As explained in the introduction of this article, the shallow water assumptions allows to approximate the free surface Navier-Stokes equations under the form of a hyperbolic conservation law, the Saint-Venant system with topography source term, that we recall hereafter
\begin{subequations}
\begin{numcases}{}
 \frac{\partial H}{\partial t} + \frac{\partial}{\partial x}
\bigl(Hu\bigr) = 0, \label{saintvenantH}\\
 \frac{\partial (Hu)}{\partial t} + \frac{\partial
  (Hu^2)}{\partial x} + \frac{g}{2}\frac{\partial
  H^2}{\partial x} = -gH \frac{\partial z_b}{\partial x}, 
\label{saintvenantHu}
\end{numcases}
\end{subequations}
where $H$ and $u$ respectively denote the water depth and
the averaged velocity and $z_b$ represents the bathymetry with
$H=\eta-z_b$. We point out the difficulty of treating a nonlinear hyperbolic system, in comparison with scalar conservation laws. A review of control theory on the Saint-Venant system is available in \cite{coron2009control}, which uses non trivial tools like the return method, quasi static deformations, and leaves the reader with still open problems.

We call entropy solution to the Saint-Venant system,
 a weak solution which satisfies an entropy inequality.
 \begin{thm}
 The system (\ref{saintvenantH})-(\ref{saintvenantHu}) is strictly hyperbolic for
 $H > 0$. It admits a mathematical entropy $\zeta$ (which is also the energy) and an entropy flux $G$
 \begin{equation}
 \zeta = \frac{Hu^2}{2}+\frac{gH(\eta+z_b)}{2}, \quad G = u\bigl(\zeta+g\frac{H^2}{2}\bigr),
 \label{saintvenantE}
 \end{equation}
 satisfying
 \begin{eqnarray}
 \frac{\partial \zeta}{\partial t} + \frac{\partial G}{\partial x}
 \leq 0.\label{eqsaintvenantE}
 \end{eqnarray}
 \end{thm}
 We do not prove this theorem which relies on the classical theory of hyperbolic
 equations and simple algebraic calculation, see \cite{dafermos2005}. We just
 recall that for smooth solutions the inequality in (\ref{eqsaintvenantE}) is
 an equality.

The system (\ref{saintvenantH})-(\ref{eqsaintvenantE}) admits a kinetic representation. We use another interpretation as the one given in the introduction. As explained in \cite{perthame1990boltzmann}, many Gibbs equilibria can indeed replace the Maxwellian usually used in kinetic theory, as soon as it satisfies basic properties. Some are particularly suited for the design of efficient numerical schemes as they enable analytical integrations (see \cite{Audusse2010a}, \cite{Audusse2010b}).

We first introduce a real function $\Chi$ defined on
$\mathbb{R}$, compactly supported and which has the properties
\begin{equation}
\left\{\begin{array}{l}
\Chi(-w) = \Chi(w) \geq 0, \\
\displaystyle \int_{\mathbb{R}} \Chi(w)\ dw = \int_{\mathbb{R}} w^2\Chi(w)\ dw = 1,
\end{array}\right.
\label{eq:chi1da}
\end{equation}
define $k_3 = \int_{\mathbb{R}} \Chi^3(w)\ dw$. We denote
$[-w_\Chi,w_\Chi]$ the compact support of length $2w_\Chi$ of the
function $\Chi$. The simplest choice for $\Chi$ is
\begin{equation*}
\Chi_1(z)=\frac{1}{2\sqrt{3}} 1_{|z| \leq \sqrt{3}}.
\label{eq:chi11}
\end{equation*}
Another function satisfying \eqref{eq:chi1da} deserves being mentioned too
\begin{equation}
\Chi_0(z)=\frac{1}{\pi} \sqrt{1 - \frac{z^{2}}{4}}\1_{\abs{z}\leq2}.
\label{eq:chi10}
\end{equation}
It is indeed the minimizer of energy. We recall the lemma from \cite{perthame2001kinetic} in a slightly different form
\begin{lem}
The minimization of the energy
\[
\zeta(f, z_{b}) = \int_{\R}{\left(\frac{\xi^{2}}{2}f(\xi) + \frac{g^{2}}{8}f(\xi)^{3} + g z_{b}f(\xi) \right)d\xi}, 
\]
under the constraints
\[
f \geq 0, \quad, \int_{\R}{f(\xi)d\xi} = H, \quad \mbox{ and } \int_{\R}{\xi f(\xi)d\xi} = Hu, 
\]
is reached for the function $M(t,x,\xi) = \frac{H}{c}\Chi_{0}(\frac{\xi - u}{c})$, with $\Chi_{0}$ defined by \eqref{eq:chi10}.
\label{lemminenergy}
\end{lem}
It also permits analytical integrations that we will use in the further numerical experiments.
Now let us construct the density of particles $M(t,x,\xi)$
defined by a Gibbs equilibrium. The microscopic density of particles present at
time $t$, at the abscissa $x$ and with velocity $\xi$ given by
\begin{equation*}
M(t,x,\xi) = \frac{H}{c} \Chi\left(\frac{\xi - u}{c}\right), 
\label{eq:M}
\end{equation*}
with $c = \sqrt{\frac{gH}{2}}$.

Then we have the following theorem, whose proof can be found in \cite{Audusse2010b}
\begin{thm}
The functions $(H,u)$ are strong solutions of the Saint-Venant system
described by \eqref{saintvenantH}-\eqref{eqsaintvenantE} if and only if
 the equilibrium $M(t,x,\xi)$ is solution of the kinetic
 equation
\begin{equation}
\qquad \frac{\partial M}{\partial t} + \xi \cdot \nabla_{x} M -
g \nabla_{x} z_b \cdot \nabla_{\xi} M = Q(t,x,\xi), 
\label{eq:gibbs}
\end{equation}
where $Q(x,t,\xi)$ is a collision term satisfying
\begin{equation*}
\int_{\mathbb{R}}  Q\ d\xi = \int_{\mathbb{R}}  \xi Q\ d\xi = 0.
\label{eq:collisionda}
\end{equation*}
The solution is an entropy solution if additionally
\begin{equation*}
\int_{\mathbb{R}} \left(\frac{\xi^2}{2} + \frac{g^2}{8k_3}M^2 + gz_b\right) Q  d\xi  \leq 0.
\label{eq:calc_kin}
\end{equation*}
\label{thm:kinetic_sv}
\end{thm}
\begin{rem}
 The right hand side $Q(t,x,\xi)$ in the kinetic representation does not vanish in the smoothness region, whatever the choice of $\Chi$, as opposed to the measure $m$ in Section~\ref{subsec:kinform}.
\end{rem}

\subsection{Convergence result for the Saint-Venant system}

Let us deal now with data assimilation on this Saint-Venant system with a bounded source term.
We prove that convergence towards a Saint-Venant solution is possible when observing the water height only. This is the key point of our Saint-Venant study. For that reason, we define a kinetic density in the case we only have access to water height observations
\be
\widetilde{M}(t,x,\xi) = \frac{H}{c}\chi\left(\frac{\xi-\widehat u}{c} \right),
\label{dagibbsM}
\ee
where $H(t,x)$ are the observations and $\widehat u(t,x)$ is the velocity of macroscopic estimator (with the notations ). As a remark, we claim that the development hereafter is adaptable to the case we only have access to $u(t,x)$.
At the kinetic level, the equation satisfied by the estimator density $f$ can be written as follows
\be
\partial_t f(t, x, \xi) + \xi \cdot \nabla_x f(t, x, \xi) - g \nabla_x z_b \nabla_{\xi} f(t, x, \xi) = \lb (\widetilde{M}(t, x, \xi) - f(t, x, \xi)),
\label{DASVST}
\ee
where $z_{b}$ is the topography. Let us suppose that the solution
$f$ of \eqref{DASVST} can also be written under the form of a Gibbs equilibrium using the macroscopic estimator $(\widehat H(t,x), \widehat u(t,x))$ i.e.
\be
f(t, x, \xi) = \frac{\widehat H}{\widehat c}\Chi (\frac{\xi-\widehat u}{\widehat c}).
\label{dagibbsf}
\ee
Integrating in $\xi$, \eqref{DASVST} with \eqref{dagibbsf} corresponds to the macroscopic system (to be compared with \eqref{mass} - \eqref{momentum})
\begin{subequations}
	\begin{numcases}{}
        \partial_t \hat H + \partial_x (\hat H\hat u) = \lb( H - \hat H), \label{mass2}\\
        \partial_t \hat H\hat u + \partial_x (\hat H\hat u^2 + \frac{g\hat H^2}{2}) = \lb \hat u \left( H - \hat H\right). \label{momentum2}
	\end{numcases}
\end{subequations}
\begin{prop}
Let $H(t,x) \in L^{\infty}([0,T]\times \R_{x})$ be solution of \eqref{saintvenantH} - \eqref{saintvenantHu}, and $\widetilde{M}(t,x,\xi)$ be the Gibbs equilibrium built from it following \eqref{dagibbsM}. Let $f(t,x,\xi)$ satisfy \eqref{DASVST}, with the constraint of being of the form \eqref{dagibbsf}. Assuming some regularity on $\widetilde{M}$ defined by
\[
\partial_{t} \widetilde{M} + \xi \cdot \nabla_{x}\widetilde{M} + g\frac{\partial z_b}{\partial x} \left(\frac{\xi - u}{c^2}\right)\widetilde{M}= Q \in C([0,T],V_{1}), 
\]
then, $\exists \gamma > 0$ such that $\forall T > 0$
\begin{equation*}
\norm{f(T) - \widetilde{M}(T)}_{V_{1}} \le \norm{f_0 -
\widetilde{M_{0}}}_{V_{1}}e^{-(\lb - \gamma)T} + \sup_{0 \leq t \leq T}\norm{Q(t)}_{V_1} \frac{1 -
e^{-(\lb - \gamma)T}}{\lb - \gamma}.
\label{resultDASVST}
\end{equation*}
\label{propDASV}
\end{prop}
\begin{rem}
As in the scalar case, the upper bound consists of two terms. The first implies the initial conditions and disappears with time, as soon as $\lb > \gamma$. The second is the filter on the collision term, which is useful once again as soon as $\lb > \gamma$.
\end{rem}
\begin{proof}[Proof of Proposition \ref{propDASV}]
To begin the proof, we introduce another data assimilation problem, which is equivalent to \eqref{DASVST}
\be
\partial_t f(t, x, \xi) + \xi \cdot \nabla_x f(t, x, \xi) - g \nabla_x z_b \nabla_{\xi} \overline{f}(t, x, \xi) = \lambda (\widetilde{M}(t, x, \xi) -
f(t, x, \xi)),
\label{DASVST2}
\ee
where $\overline{f}$ is defined hereafter
\begin{equation*}
\overline{f}(t, x, \xi) = \frac{\widehat H}{\widehat c}\overline{\Chi} (\frac{\xi-\widehat u}{\widehat c}),
\label{dagibbs2}
\end{equation*}
and $\overline{\Chi}$ satisfies
\begin{equation*}
\overline{\Chi}(a) = \overline{\Chi}(-a),\qquad \int{\overline{\Chi}(z)dz} = 1.
\label{propchibarre}
\end{equation*} 
Given the properties of $\overline{\Chi}$, an integration of \eqref{DASVST} and
\eqref{DASVST2} in the $\xi$ variable results in the same macroscopic equations.
Therefore working with \eqref{DASVST2} also leads to the right data assimilation
problem. We choose the particular expression for $\overline{\Chi}$:
\begin{equation*}
\overline{\Chi} = \int_z^{+\infty}{\xi \Chi(\xi)d\xi}.
\label{chibarre}
\end{equation*}
Hence, \eqref{DASVST2} simplifies into
\begin{equation*}
\partial_t f(t, x, \xi) + \xi \cdot \nabla_x f(t, x, \xi) + g \nabla_x z_b\frac{\xi - u}{c^2} f(t, x, \xi) = \lb (\widetilde{M}(t, x, \xi) -
f(t, x, \xi)).
\label{DASVST3}
\end{equation*}
From the hypothesis on $\widetilde{M}(t,x,\xi)$ we have
\be
\partial_t (f - \widetilde{M}) + \xi \cdot \nabla_x(f - \widetilde{M}) + (\lb +
g\frac{\partial z_b}{\partial x} \frac{\xi - u}{c^2})(f - \widetilde{M}) = -Q.
\label{DASVST4} 
\ee
Multiplying \eqref{DASVST4} by $\sgn (f(t, x, \xi) - \widetilde{M}(x, \xi))$ and assuming the source term is bounded leads to
\be
\partial_t \abs{f - \widetilde{M}} + \xi \cdot \nabla_x\abs{f - \widetilde{M}} + (\lb
- g\norm{\frac{\partial z_b}{\partial x}}_{L^{\infty}(\R_{x}^{d})} \frac{\abs{\xi -
u}}{c^2})\abs{f - \widetilde{M}} \leq \abs{Q}.
\label{DASVST6} 
\ee
Integrating \eqref{DASVST6} over the spatial and velocity variables we end up with
\begin{equation*}
\frac{d}{dt} \norm{f - \widetilde{M}}_{V_{1}}
 + \int_{R_x}{\int_{R_{\xi}}{(\lb - g\norm{\frac{\partial z_b}{\partial
x}}_{\infty} \frac{\abs{\xi - u}}{c^2})\abs{f - \widetilde{M}}d\xi dx}}
 \leq \int_{R_x}{\int_{R_{\xi}}{\abs{Q}d\xi}dx}.
\label{DASVST7} 
\end{equation*}
Since $\widetilde{M}$ and $f$ are considered to be Gibbs equilibria of compact
support in $\xi$ we introduce $\Omega$ the support of $M$ and
$\xi_{max}$ defined such that
$$\forall \xi \mbox{ s.t. } \abs{\xi}\ge \xi_{max}, \ f(t, x, \xi) =
\widetilde{M}(t,x,\xi) = 0. $$
Besides, we consider that $u(t,x) \in L^{\infty}([ 0,T ],R_x^{d})$, $Q
\in V_{1}$. We can
therefore proceed further
\begin{equation*}
\begin{aligned}
\frac{d}{dt} \norm{f - \widetilde{M}}_{V_{1}}
 + (\lb - g\norm{\frac{\partial z_b}{\partial x}}_{\infty} \frac{\xi_{max} +
\norm{u}_{\infty}}{c^2})\norm{f - \widetilde{M}}_{V_{1}} \leq \norm{Q(t)}_{V_1}.
\end{aligned}
\label{DASVST8} 
\end{equation*}
Eventually,
\begin{equation*}
\norm{f - \widetilde{M}}_{V_{1}} \le \norm{f_0 -
\widetilde{M_{0}}}_{V_{1}}e^{-(\lb - \gamma)t} + \sup_{0 \leq t \leq T}\norm{Q(t)}_{V_1} \frac{1 -
e^{-(\lb - \gamma)t}}{\lb - \gamma},
\label{DASVST9}
\end{equation*}
where we have set
$$\gamma = g\norm{\frac{\partial z_b}{\partial x}}_{\infty} \frac{\xi_{max} +
\norm{u}_{\infty}}{c^2}.$$
\end{proof}
In a second step, we prove that the observation of $H$ is enough to have convergence of $f$ towards $M$, where $M$ is defined as follows
\[
M = \frac{H}{c}\chi\left(\frac{\xi- u}{c} \right). 
\]
This means that $M$ is entirely built with observations (unlike $\widetilde{M}$) and we want to know if the problem defined by \eqref{DASVST} leads to a convergence of $(\widehat H, \widehat u)$ towards $(H,u)$. This is the key point of the study on the Saint-Venant system. 
\begin{prop}
Under the hypothesis of Prop.\ref{propDASV}, we have $\forall$ $T>0$
\begin{equation*}
\norm{f - M}_{V_{1}} \xrightarrow[\lb \to +\infty]{}0.
\end{equation*}
\label{propsaintvenant2}
\end{prop}
\begin{proof}[Proof of Proposition \ref{propsaintvenant2}]
The triangular inequality tells us that
\[
\norm{f - M}_{V_{1}}\leq\norm{f - \widetilde{M}}_{V_{1}} + \norm{\widetilde{M} - M}_{V_{1}}. 
\]
We have already proved in Prop.\ref{propDASV} that $\norm{f - \widetilde{M}}_{V_{1}} \xrightarrow[\lb \to +\infty]{} 0$. It is clear that $\widetilde{M}$ is also one solution to the data assimilation problem
\[
\frac{\partial \widetilde{M}}{\partial t} + \xi \cdot \nabla_{x}  \widetilde{M}-
g\frac{\partial z_b}{\partial x}\frac{\partial \widetilde{M}}{\partial
  \xi} = \lb(M-\widetilde{M}), 
\]
Subtracting it with \eqref{eq:gibbs} gives
\begin{equation}
\frac{\partial \delta}{\partial t} + \xi \cdot \nabla_{x}  \delta -
g\frac{\partial z_b}{\partial x}\frac{\partial \delta}{\partial
  \xi} = -\lb \delta,
\label{eq:delta}
\end{equation}
where $\delta = \widetilde{M} - M$. And since
$$\int_{\mathbb{R}} \delta(t,x,\xi)\ d\xi = H(t,x,\xi) - H(t,x,\xi) = 0,$$
the integration in $\xi$ of~\eqref{eq:delta} gives necessarily
$u=\widehat u$ and therefore $\widetilde{M} = M$ (as soon
as $u$ is known at one point e.g. at one of the boundaries). And the
proof is complete.
\end{proof}
\begin{rem}
In other words, if $(H,u)$ is a solution to the Saint-Venant system, there are no other solutions of the form $(H,\widehat u)$ with $u\neq\widehat u$.
\end{rem}
\label{subsec:saintvenant}

\subsection{Numerical results}
\label{subsec:svtnumerics}
Let us start first by some reminder of the techniques used for an appropriate discretization of the shallow water equations without assimilation. To approximate the solution of the Saint-Venant system~\eqref{saintvenantH}-\eqref{saintvenantHu}, we use a finite volume framework.
We assume that the computational domain is discretised by $I$ nodes $x_i$.
  We denote $C_i$ the cell of length $\Delta
x_i=x_{i+1/2}-x_{i-1/2}$ with $x_{i+1/2}=(x_i+x_{i+1})/2$. For the time discretization, we denote $t^n = \sum_{k \leq n} \Delta t^k$ where the time steps $\Delta t^k$ will be
precised later though a CFL condition. The ratio between the space and
time steps is  $\sigma^n_i=\Delta t^n/\Delta x_i$. We denote
$X^n_{i}=(H^n_i,q^n_{i})$ the approximate solution at time $t^n$ on
the cell $C_i$ with $q^n_{i}=H^n_i u^n_{i}$.

\subsubsection{The hydrostatic reconstruction technique}

The hydrostatic reconstruction consists in a modification of the
bottom topography and the water depths at the cell interfaces, see
\cite{bristeau1}. More precisely the quantities $H^n_{i+1/2\pm}$ are
defined by
\begin{equation}
H^n_{i+1/2-} = H^n_i + z_{b,i} - z_{b,i+1/2},\quad H^n_{i+1/2+} =
H^n_{i+1} + z_{b,i+1} - z_{b,i+1/2},
\label{eq:def_rec}
\end{equation}
with
\be
z_{b,i+1/2}  = \max(z_{b,i},z_{b,i+1}),\label{eq:zb_rec}
\ee
A nonnegativity-preserving truncation of the leading order
depths is also applied in~\eqref{eq:def_rec}
\begin{equation}
H^n_{i+1/2\pm} = \max(0,
H^n_{i+1/2\pm}).
\label{eq:pos_rec}
\end{equation}

\subsubsection{Discrete scheme}
\label{subsec:discrete_scheme}

To precise the numerical scheme, we deduce 
a finite volume kinetic scheme from the kinetic
interpretation~\eqref{eq:gibbs} of the Saint-Venant system.
First, we define the discrete densities of particles $M_{i}^n$ by	
\begin{equation}
M^n_i = M_{i}^n (\xi)= \frac{H^n_{i}}{c^n_i} \chi\left(\frac{\xi -
    u^n_{i}}{c^n_i}\right),
\label{eq:gibbs_d}
\end{equation}
with $c^n_i= \sqrt{ \frac{gH^n_{i}}{2}}$. Then we propose a discretisation of~\eqref{eq:gibbs} under the form
\begin{eqnarray}
M_{i}^{n+1-} & = & M_{i}^{n} - \sigma^n_i \left( {\mathcal M}^n_{i+1/2}  -
  {\mathcal M}_{i-1/2}^{n} \right),
  \label{eq:cindis0}
\end{eqnarray}
with
\begin{eqnarray*}
{\mathcal M}^n_{i+1/2} & = &  \xi M_{i+1/2}^{n} - g \Delta
z_{b,i+1/2-}\frac{\partial \overline{M}^n_{i+1/2-}}{\partial \xi},\\
{\mathcal M}^n_{i-1/2} & = &  \xi M_{i-1/2}^{n} - g \Delta
z_{b,i-1/2+}\frac{\partial \overline{M}^n_{i-1/2+}}{\partial \xi},
\end{eqnarray*}
and
\begin{eqnarray*}
M^n_{i+1/2} & = & M^n_{i+1/2-}\1_{\xi\geq 0} + M^n_{i+1/2+}\1_{\xi\leq
  0},\mbox{ with } M^n_{i+1/2-}  =  \frac{H^n_{i+1/2-}}{c^n_{i+1/2-}} \chi\left(\frac{\xi -
    u^n_{i}}{c^n_{i+1/2-}}\right),\\
\Delta z_{b,i+1/2-} & = & z_{b,i+1/2} - z_{b,i},\quad \Delta z_{b,i-1/2+}  =  z_{b,i-1/2} - z_{b,i}, \\
\overline{M}^n_{i+1/2-} & = & \frac{H^n_i + H^n_{i+1/2-}}{2c^n_{i+1/2-}} \chi\left(\frac{\xi -
    u^n_{i}}{c^n_{i+1/2-}}\right),\quad \overline{M}^n_{i-1/2+}  =  \frac{H^n_i + H^n_{i-1/2+}}{2c^n_{i-1/2+}} \chi\left(\frac{\xi -
    u^n_{i}}{c^n_{i-1/2+}}\right).
\end{eqnarray*}

The discrete scheme (\ref{eq:cindis0}) does not take
into account the collision term which relaxes $M_{i}^{n+1-}$
to a Gibbs equilibrium. It is used in a second step introducing a discontinuity at time
$t^{n+1}$ on $M$ and replacing $M_{i}^{n+1-}$ by an equilibrium
\[
M_i^{n+1} = M_i^{n+1 - } + \Delta t^n Q_i^n,  
\]
such that $M_i^{n+1}$ is supposed to be of the form \eqref{eq:gibbs_d}
where $H^{n+1}_i$ and $u^{n+1}_i$ are computed using an integration of~\eqref{eq:cindis0}
\begin{equation}
X_i^{n+1} = \left( \begin{array}{c} H_i^{n+1}\\ H_i^{n+1}u_i^{n+1}
 \end{array} \right) = \int_\R {\mathcal K}(\xi) M_{i}^{n+1-} (\xi)\
d\xi,
\label{eq:Xi}
\end{equation}
with ${\mathcal K}(\xi)$ is the vector ${\mathcal K}(\xi) = (1,\xi)^T$.
Notice that $M_{i}^{n+1}$ is discontinuous in the sense that
$M_{i}^{n+1} \neq M_{i}^{n+1-}$ whereas the macroscopic variables
remain continuous $X_i^{n+1} = X_i^{n+1-}$.

\subsubsection{Fluxes calculus}
\label{subsec:fluxes}
Using (\ref{eq:Xi}), we are now able to precise the numerical scheme
for the Saint-Venant system \eqref{saintvenantH}-\eqref{eqsaintvenantE}
\begin{equation}
X_i^{n+1} = X_i^{n} - \sigma_i^n \left( {\mathcal F}^n_{i+1/2} - {\mathcal
    F}^n_{i-1/2} \right),
\label{eq:fvda}
\end{equation}
with
\begin{eqnarray}
& & {\mathcal F}^n_{i+1/2} = \int_\R  {\mathcal K}(\xi) {\mathcal M}^n_{i+1/2} (\xi) \ d\xi.\label{eq:fxda}
\end{eqnarray}
Now, using~\eqref{eq:cindis0}, we are able to precise the computation of the macroscopic
fluxes in \eqref{eq:fvda}. If we denote
\begin{equation*}
{\mathcal F}(X_{i+1/2}^n) = \left( {\mathcal F}_H(X_{i+1/2}^n), {\mathcal F}_{q}(X_{i+1/2}^n)\right)^T,
\label{eq:fluxda}
\end{equation*}
then we have
\begin{multline}
 {\mathcal F}_H(X_{i+1/2}^n)  =  H_{i+1/2+}^n \int_{z\leq -\frac{u_{i+1}^n}{c_{i+1/2+}^n}}
(u_{i+1}^n + z c_{i+1/2+}^n)\chi(z)\ dz \\
 + H_{i+1/2-}^n \int_{z\geq -\frac{u_{i}^n}{c_{i+1/2-}^n}}
(u_{i}^n + z c_{i+1/2-}^n) \chi(z)\ dz,\label{eq:FHda}
\end{multline}
and
\begin{multline}
{\mathcal F}_{q}(X_{i+1/2}^n)  =  H_{i+1/2+}\int_{z\leq
  -\frac{u_{i+1}^n}{c_{i+1/2+}^n}}
(u_{i+1}^n + z c_{i+1/2+}^n)^2 \chi(z)\ dz\\
 +H_{i+1/2-}^n \int_{z\geq
  -\frac{u_{i}^n}{c_{i+1/2-}^n}}
(u_{i}^n + z c_{i+1/2-}^n)^2 \chi(z)\ dz \\
  +\frac{g\Delta z_{b,i+1/2-}}{2}(H^n_{i+1/2-} + H^n_{i}).\label{eq:Fq}
\end{multline}

\subsubsection{Properties of the numerical scheme}
\label{sec:properties}

The previous scheme exactly corresponds to the original
reconstruction technique proposed by Audusse {\it et al.} \cite{bristeau1} with a
kinetic solver for the conservative part and therefore it is endowed with the properties
depicted in~\cite[theorem~2.1]{bristeau1}. We recall hereafter the statement of the theorem.
\begin{thm}
The kinetic flux for the homogenous Saint-Venant system is
consistent, preserves the nonnegativity of the water depth and satisfies an
in-cell entropy inequality corresponding to the entropy $\zeta$
in~\eqref{eqsaintvenantE}. Then the finite volume scheme~\eqref{eq:fvda}
with the definitions~\eqref{eq:def_rec},\eqref{eq:zb_rec},\eqref{eq:pos_rec},\eqref{eq:Xi},\eqref{eq:FHda}
and\eqref{eq:Fq}
\begin{itemize}
\item[{\it (i)}] preserves the nonnegativity of the water depth,
\item[{\it (ii)}]  preserves the steady state of a lake at rest
  $H+z_b=cst$,
\item[{\it (iii)}] is consistent with the Saint-Venant
  system~\eqref{saintvenantH}-\eqref{saintvenantHu},
\item[{\it (iv)}] satisfies an in-cell entropy inequality associated
  to the entropy $\zeta$ defined in~\eqref{saintvenantE}
$$\frac{d}{dt} \zeta(X_i(t),z_{b,i}) + \frac{1}{\Delta x_i} \left(G_{i+1/2} -
G_{i-1/2}\right) \leq 0.$$
The preceding inequality is a semi-discrete version of~\eqref{eqsaintvenantE}.
\end{itemize}
\label{thm:hyd_rec}
\end{thm}
Since the proof of the preceeding theorem is similar to the one given
in~\cite[theorem~2.1]{bristeau1}, it is not detailed here.

\subsubsection{Numerical scheme for the data assimilation problem}
Let us deal now with the numerical scheme for the assimilation problem \eqref{DASVST}. Since the handling of the topography source terms with hydrostatic reconstruction has been explained in the previous section, we will only consider here the source term associated to the nudging term. We are working on
\be
\partial_{t}f(t,x,\xi) + \xi \cdot \nabla_{x} f(t,x,\xi)= \lb \left( \widetilde{M} - f \right),
\label{svtnotopo}
\ee
for which we propose the two steps numerical scheme composed first of a transport stage
\be
f_{i}^{n+1-} = f_{i}^{n} - \sigma_{i}^{n} \xi \left( f_{i+1/2}^{n} - f_{i-1/2}^{n}\right) + \Delta t^{n} \lb \left( \widetilde{M}_{i}^{n} - f_{i}^{n}\right),
\label{svtkinschemenotopo}
\ee
where the numerical fluxes are computed using a simple upwind scheme
\begin{equation}
	f_{i+1/2}^n =
	\begin{cases}
	    f_i^{n}, & \text{ for } \xi \geq 0  \\
	    f_{i+1}^{n}, & \text{ for } \xi < 0; 
	\end{cases}
 \label{upwindda}
\end{equation}
and a collapse stage that forces $f_i^{n+1}$ to be of the form
\be
f_i^{n+1} = \frac{\widehat H^{n+1}_{i}}{\widehat c^{n+1}_i} \chi\left(\frac{\xi
    - \widehat u^{n+1}_{i}}{\widehat c^{n+1}_i}\right).
\label{gibbsobserver}
\ee
As in the 1D case for Burgers' equation, this leads, by an integration against $\xi$ of the first two moments, to a macroscopic scheme
\begin{align*}
&\widehat H_{i}^{n+1} = \widehat H_{i}^{n} - \sigma_{i}^{n} \left( ({\mathcal F}_H(X_{i+1/2}^n) - {\mathcal F}_H(X_{i-1/2}^n) \right) + \lb \Delta t^{n}\left( H_{i}^{n} - \widehat H_{i}^{n}\right),\\
&\widehat H_{i}^{n+1}\widehat u_{i}^{n+1}=\widehat H_{i}^{n}\widehat u_{i}^{n} 
- \sigma_{i}^{n} \left(  {\mathcal F}_q(X_{i+1/2}^n) -  {\mathcal F}_q(X_{i-1/2}^n)\right)+ \lb \Delta t^{n} \left( H_{i}^{n}\widehat u_{i}^{n} - \widehat H_{i}^{n}\widehat u_{i}^{n}\right).
\label{svtmacroschemenotopo}
\end{align*}
Let us denote 
\begin{equation*}
e(f) = \frac{\xi^{2}}{2}f + \frac{g^{2}}{8k_{3}}f^{3},
\end{equation*}
where $k_{3} = \int_{\R}{\chi^{3}(\omega)d\omega}$, 
\[
\widehat \zeta_{i}^{n} = \int_{\R}{e(f_{i}^{n})(\xi)d\xi} \quad \mbox{ and } \quad \widetilde{\zeta}_{i}^{n} = \int_{\R}{e(\widetilde{M}_{i}^{n})(\xi)d\xi}, 
\]
respectively the observer energy and the observations energy of cell $i$ at time $n$, and
\[
 G_{i+1/2}^{n} = \int_{\R}{\xi e(f_{i+1/2}^{n})(\xi)d\xi}, 
\]
 the energy flux between cell $i$ and cell $i+1$. Then we can state a deeper stability property than the positivity of the water height.
\begin{thm}
Under the CFL condition 
\be
\Delta t^{n} \leq \min_{i}\frac{\Delta x_{i}}{\lb \Delta x_{i} +\abs{\widehat u_{i}^{n}} + \widehat c_{i}},
\label{cflobserver}
\ee
the finite volume scheme \eqref{svtkinschemenotopo} with the definition \eqref{upwindda} preserves the non-negativity of the water height and admits the discrete entropy inequality:
 \begin{equation*}
 \widehat \zeta_{i}^{n+1} \leq \widehat \zeta_{i}^{n} - \sigma_{i}^{n}\left( G_{i+1/2}^{n} - G_{i-1/2}^{n}\right) + \lb \Delta t^{n} \left( \widetilde{\zeta}_{i}^{n}-\widehat \zeta_{i}^{n}\right).
 \label{discreteentropy}
 \end{equation*}
 \label{thmentropy}
\end{thm}
\begin{proof}{Proof of Theorem~\ref{thmentropy}}
Let us denote
$$\xi^{+} = \max(\xi,0) \quad \mbox{ and} \quad f^{-} = \min(\xi,0).$$
From \eqref{svtnotopo}, the scheme \eqref{svtkinschemenotopo} with \eqref{upwindda} leads to
\[
f_{i}^{n+1-} = f_{i}^{n} - \sigma_{i}^{n} \xi^{+}\left(f_{i}^{n} - f_{i - 1}^{n}\right) - \sigma_{i}^{n} \xi^{-}\left(f_{i+1}^{n} - f_{i }^{n}\right) + \lb \Delta t^{n} \left( \widetilde{M}_{i}^{n} - f_{i}^{n}\right), 
\]
which can be rewritten
\be
f_{i}^{n+1-} = f_{i}^{n}\left( 1 - \sigma_{i}^{n} \xi^{+} + \sigma_{i}^{n} \xi^{-} - \lb \Delta t^{n}\right) + \sigma_{i}^{n} \xi^{+} f_{i - 1}^{n}- \sigma_{i}^{n} \xi^{-}f_{i+1}^{n}  + \lb \Delta t^{n} \widetilde{M}_{i}^{n}. 
\label{scheme1}
\ee
Under the condition that 
\[
1 - \sigma_{i}^{n} \abs{\xi} - \lb \Delta t^{n} \leq 0,
\]
which can be written under the form of the CFL condition \eqref{cflobserver},
the scheme \eqref{scheme1} is a convex combination of $f_{i}^{n}, f_{i-1}^{n}, f_{i+1}^{n}$ and $M_{i}^{n}$.  As a consequence, we can claim that if the initial condition is non-negative, and so are the observations, the density $f_{i}^{n}$ remains positive for any time, \textit{a fortiori} the water height. The scheme preserves the non-negativity of the water height. In the case where $f \geq 0$, it is a convex function, which enables us to write for the convex combination \eqref{scheme1}
\begin{multline}
e(f_{i}^{n+1-}) \leq  \left( 1 - \sigma_{i}^{n} \xi^{+} + \sigma_{i}^{n} \xi^{-} - \lb \Delta t^{n}\right)e(f_{i}^{n}) \\ + \sigma_{i}^{n} \xi^{+} e(f_{i - 1}^{n})- \sigma_{i}^{n} \xi^{-}e(f_{i+1}^{n})  + \lb \Delta t^{n} e(\widetilde{M}_{i}^{n}).
\label{scheme2}
\end{multline}
And \eqref{scheme2} is factorized again as
\be
e(f_{i}^{n+1-}) \leq e(f_{i}^{n}) - \sigma_{i}^{n} \xi\left(e(f_{i+1/2}^{n}) - e(f_{i - 1/2}^{n})\right) + \lb \Delta t^{n} \left( e(\widetilde{M}_{i}^{n}) - e(f_{i}^{n})\right).
\label{scheme3}
\ee
 We integrate \eqref{scheme3} with respect to $\xi$ to get
 \[
 \widehat \zeta_{i}^{n+1-} \leq \widehat \zeta_{i}^{n} + \sigma_{i}^{n}\left( G_{i+1/2}^{n} - G_{i-1/2}^{n}\right) + \lb \Delta t^{n} \left( \widetilde{\zeta}_{i}^{n}-\widehat \zeta_{i}^{n}\right). 
 \]
Eventually, Lemma~\ref{lemminenergy} ensures that the minimum of energy is achieved for the Gibbs equilibrium of the form \eqref{gibbsobserver}, authorizing us to write
 \[
  \int_{\R}{e(f_i^{n+1})d\xi} = \widehat \zeta_{i}^{n+1}  \leq \widehat \zeta_{i}^{n+1-} \leq \widehat \zeta_{i}^{n} - \sigma_{i}^{n}\left( G_{i+1/2}^{n} - G_{i-1/2}^{n}\right) + \lb \Delta t^{n} \left( \widetilde{\zeta}_{i}^{n}-\widehat \zeta_{i}^{n}\right). 
 \]
\end{proof}
\begin{rem}
 In case of regular solutions, it can be shown that $\widehat \zeta$ tends to $\zeta$ in time, echoing Proposition \ref{thmkineasy}.
\end{rem}

\subsubsection{Simulations}
In order to show how the kinetic Luenberger observer converges to non trivial equilibria, the first test case in two dimensions are the Thacker solutions \cite{thacker1981some}. Those are analytical periodic solutions to the Saint-Venant system in a parabolic bowl for which the free surface always remains a straight line. They are quite difficult to handle numerically, because of the presence of numerous drying and flooding areas in time. The domain is thus the onedimensional parabolic bowl characterized by 
$$z_b (x) = \frac{h_{m}}{a^{2}}\left(\left(x-\frac{L}{2}\right)^2-a^{2}\right),$$
where $a = 1.0$, $L = 4.0$ and $h_m = 0.5$. Our reference solution, which will serve as observations, starts with an initial water height $$H(0,x) = \max(0,-h_{m}/a^2((x+1/2)^2-a^2)),$$ and null velocity $u(0,x) = 0$. This makes the surface a straight line of equation $\eta(x) = \frac{-h_m}{a^2}(4x - 0.25)$. The assimilated problem on the contrary starts from an equilibria $$\widehat H(0,x) = \max(0,-z_b (x)),$$ $u(0,x)=0$, such that the surface has equation $\eta(x) = 0$. The domain is divided into 300 cells and the simulation lasts 15 seconds. Exact observations are carried out over the spatial interval $\left[ 1.5 , 2.5\right]$ ($25\%$ of the domain) and every $0.05$s. Figure~\ref{thackerevolution} shows snapshots of the evolution of an observer towards an analytical Thacker solution. Visually, the observer converges quite fast towards the observations. To be more specific, we study the relative error between the exact height and the observer's 
height. In order to illustrate the different convergence 
results with 
respect to $\lambda$, we repeat the experiment for several values of the nudging coefficient. Figure~\ref{thackererror} on the left plot shows the result in the case where observations are exact, but partial in space and time. We obtain two rapidly decreasing curves, which goes approximately to zero when $\lambda$ gets bigger (the residual error is due to the discretization of the numerical scheme). This is what was expected from Proposition~\ref{propDASV}: without noise, the bigger $\lambda$, the better. However, this assessement is not true numerically as $\lb$ influences badly the time step \eqref{cflobserver}. Figure~\ref{thackererror} shows the same situation with the right plot, when the noise already used in Section~\ref{resultsburgers} is added to data ( $\varepsilon^{1 - \alpha}\cos(\frac{x}{\varepsilon})$, with $\varepsilon = 0.1$ and $\alpha = 0.25$). Even if Theorem~\ref{thmnoise1lin} did not apply to systems of conservation laws with source terms, we see clearly here that there is an optimal $\
lambda$ (here around 13), balancing 
the 
fact that the gain improves the match with data but also amplifies the noise. We notice also that the $L^{1}$ error norm does not tend to zero, since it cannot get rid of the perturbations added on the data.
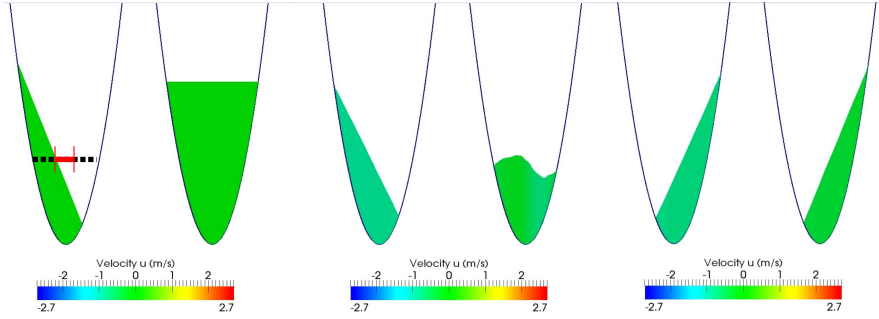
\begin{figure}[h!]
\begin{center}
\begin{tikzpicture}
\begin{pgfonlayer}{background}
\node{\includegraphics[width=0.9\textwidth]{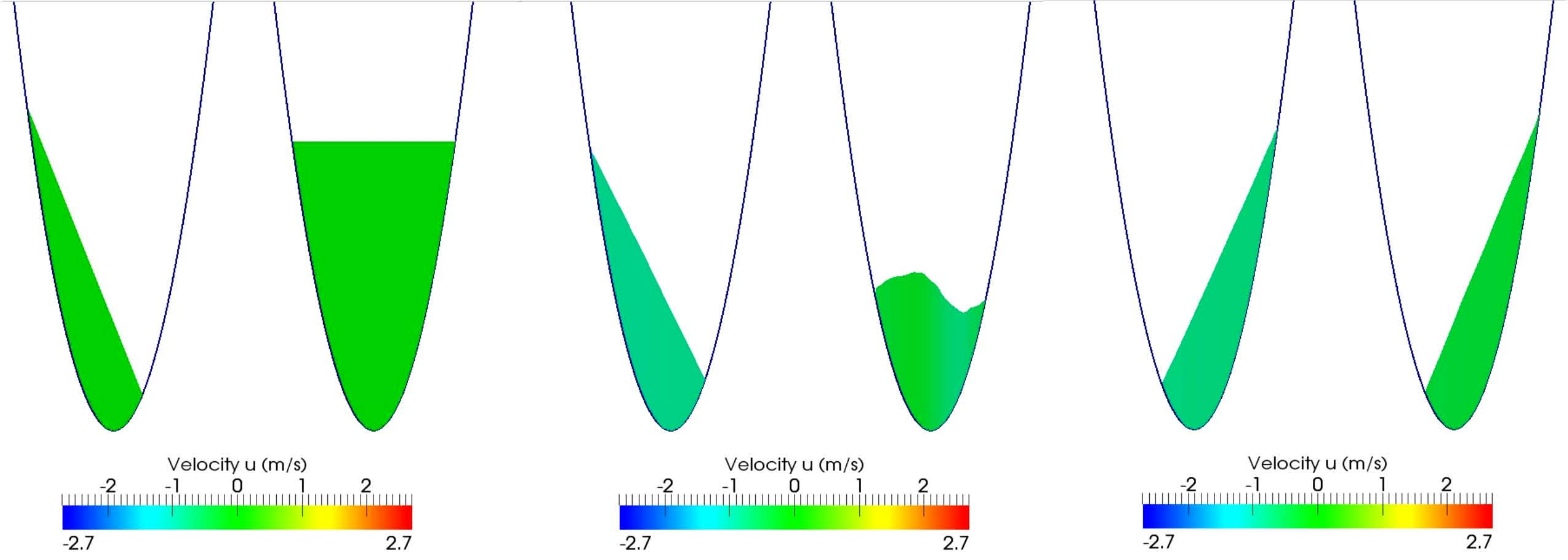}};
\end{pgfonlayer}
\draw[line width=2,densely dotted] (-5.45,0) -- (-4.60,0);
\draw[line width=2,-,red] (-5.15,0) -- (-4.90,0);
\draw[line width=2,-,red] (-5.15,0) node {$|$} ;
\draw[line width=2,-,red] (-4.90,0) node {$|$} ;
\end{tikzpicture}
\caption[Evolution of Thacker solution with assimilation]{Evolution of Thacker solution at $t=0$s, $t=2$s and $t=5$s. For each time, the left plot is the reference solution and the right plot is the observer. The red portion of the domain on the left plot is the observation domain.}
\label{thackerevolution}
\end{center}
\end{figure}

\begin{figure}[h!]
\begin{center}
\begin{tikzpicture}[scale=0.7]
\begin{axis}[
xlabel=$\lambda$,
ylabel=Relative error (\%),
title={Error at 15 s}]]
\addplot[color=red,mark=*] table[x=gain,y=error]
	 {FigDA/thacker1.txt};
	 \addlegendentry{300 cells}	
\addplot[color=blue,mark=*] table[x=gain,y=error]
	 {FigDA/thacker2.txt};	
	 \addlegendentry{500 cells}	 
\end{axis}
\end{tikzpicture}
\hspace{0.5cm}
  \begin{tikzpicture}[scale=0.7]
\begin{axis}[
xlabel=$\lambda$,
ylabel=Relative error (\%),
title={Error at 15 s}]]
\addplot[color=red,mark=*] table[x=gain,y=error]
	 {FigDA/thacker3.txt};
	 \addlegendentry{$\epsilon=0.1$}		 
\end{axis}
\end{tikzpicture}

  \caption[Thacker noisy solution relative error]{Thacker noisy solution. Relative error $\frac{\norm{\widehat H - H}_{L^{1}}}{\norm{H}_{L^{1}}}$ for different values of $\lambda$. Left: in the case where no noise is added to the observations. Right: in the case where a noise is added to the observations.}
\label{thackererror}
 \end{center}
\end{figure}
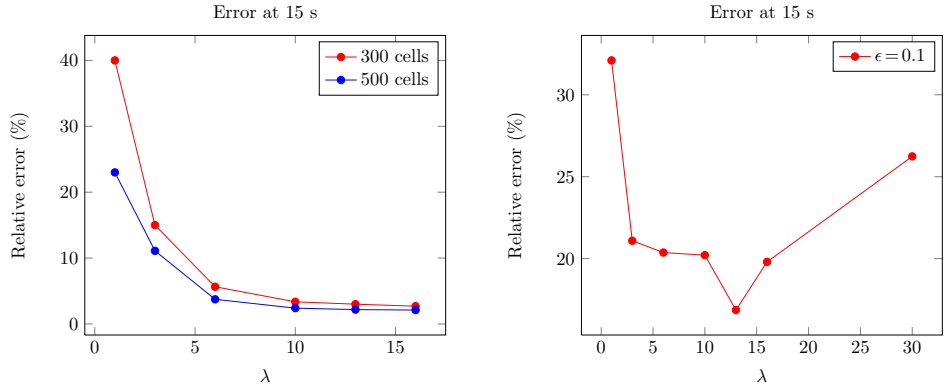

Then, we propose a numerical test of our kinetic Luenberger observer on the non-hydrostatic Saint-Venant system \cite{saintemarievertical2009}
\begin{subequations}
\begin{numcases}{}
\frac{\partial H}{\partial t} + \frac{\partial}{\partial x} \bigl(H u\bigr) = 0, \label{eq:euler_1}\\
\frac{\partial}{\partial t}(Hu) + \frac{\partial}{\partial x}\left(H u^2 
+ \frac{g}{2}H^2 + Hp_{nh,0}\right) = - (gH +
\left. p_{nh,0}\right|_b)\frac{\partial
  z_b}{\partial x} + p^a\frac{\partial \eta}{\partial x},\label{eq:euler_2}\\
\frac{\partial}{\partial t}(H w)  + \frac{\partial}{\partial
 x}(H w u) = \left. p_{nh,0}\right|_b,\label{eq:euler_3}\\
\frac{\partial}{\partial t}\left(\frac{\eta^2 - z_b^2}{2}\right)  + \frac{\partial}{\partial x}\left(\frac{\eta^2-z_b^2}{2}u\right) = H w,\label{eq:euler_4}
\end{numcases}
\end{subequations}
detailed in \cite{Bristeau201151} and recalled here without proof.
No theoretical work has yet been pursued by the authors about data assimilation on this system, but we extended the result on the observer built for the Saint-Venant system, where only the water height was required. The context is now a bidimensional 30m long pond, with the bottom topography
\begin{equation*}
	z_b(x)=
	\begin{cases}
		 z_0, & \text { for } x \leq 6 \\
		 z_0+\frac{z_M-z_0}{6}(x-6), & \text { for } 6 \leq x < 12  \\
		z_M & \text { for } 12 \leq x \leq 14  \\
		z_M+\frac{z_0-z_M}{3}(x-14) & \text { for } 14 \leq x\leq 17  \\
		z_0 & \text { for } 17 \leq x \leq 30,	
	\end{cases}
 \label{toposimu}
\end{equation*}
where $z_0 = -0.4$, $z_M =-0.1$. The domain is divided into 400 cells, and the simulation lasts $20$s. Observations are made on the whole domain in space and time, but only on $H(t,x)$. Unlike what was done for Thacker solutions, an interpolation in time is carried out between two measurements to provide an observation value every 0.05s. Figure~\ref{svtsnapshots} shows the compared evolution of the reference simulation on the left, and of the observer, on the right. Both simulations started with different initial conditions. The reference initial condition is, for $x_{\text{max}} = 30$ and $\eta_1 = 0.2$.
\begin{subequations}
\begin{align*}
 & H_0(x) = \eta_1-z_b(x) + 0.3 e^{\displaystyle -(x-0.75 x_{\text{max}})^2}, \\
 & Hu_0(x) = 0.3 \sqrt{g(\eta_1-z_b(x))} e^{\displaystyle -(x-0.75 x_{\text{max}})^2}, 
\end{align*}
\end{subequations}
While the observer begins with a completely different water height, for $\eta_2 = 1$ this time
\begin{subequations}
\begin{align*}
 & H_0(x) = \eta_2-z_b(x),  \\
 & Hu_0(x) = 0.3 \sqrt{g(\eta_1-z_b(x))} e^{\displaystyle -(x-0.75 x_{\text{max}})^2}.
\end{align*}
\end{subequations}

\begin{figure}[h!]
\begin{center}
 \begin{tabular}{cc}
\includegraphics[width=0.45\textwidth]{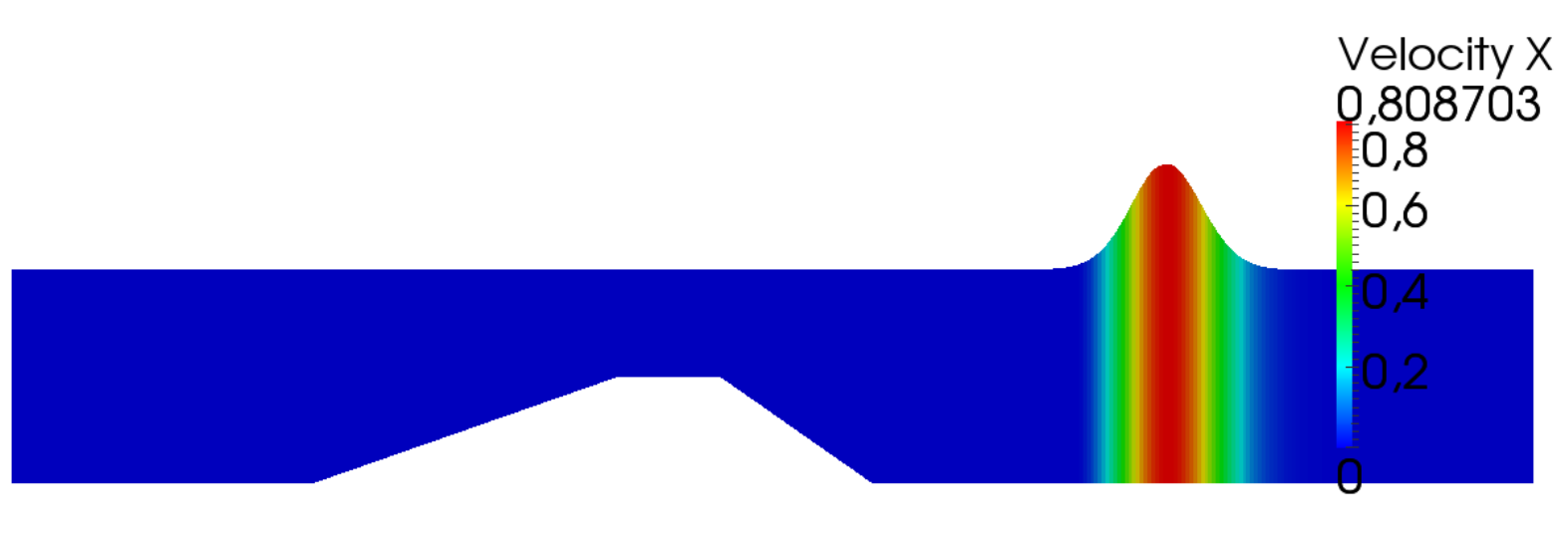} & \includegraphics[width=0.45\textwidth]{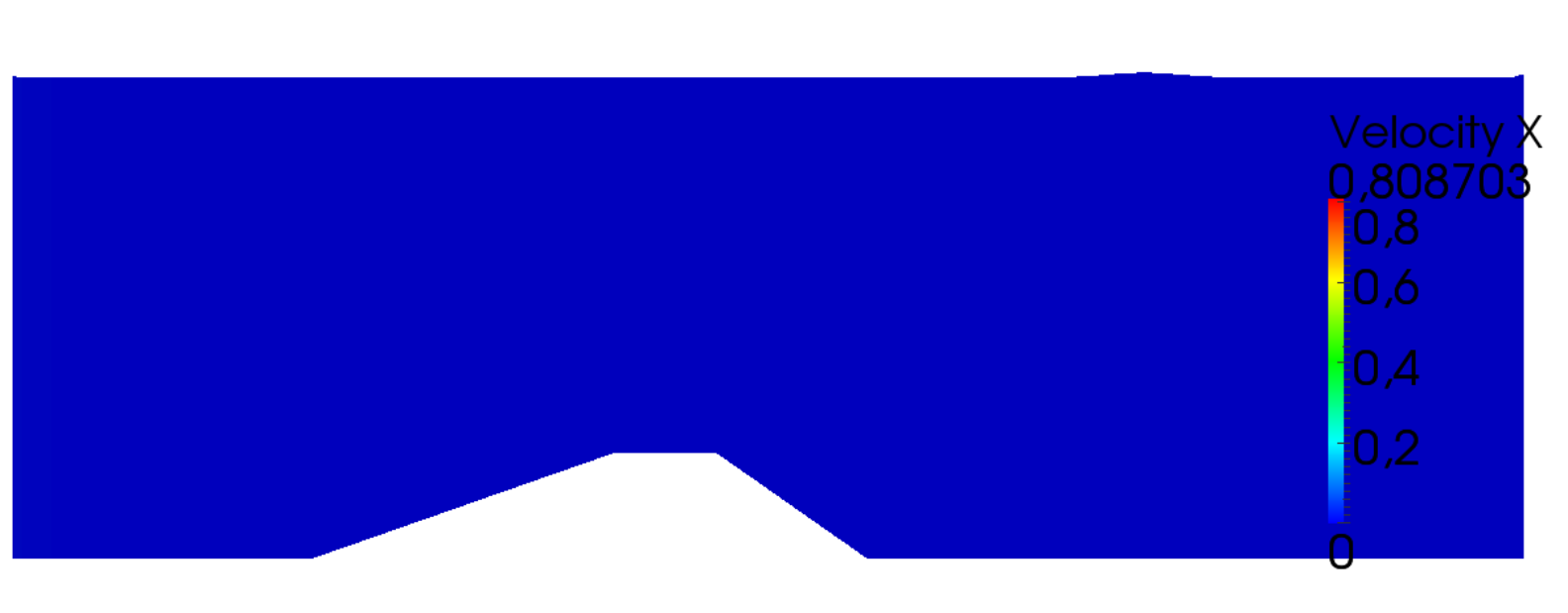}\\
 \includegraphics[width=0.45\textwidth]{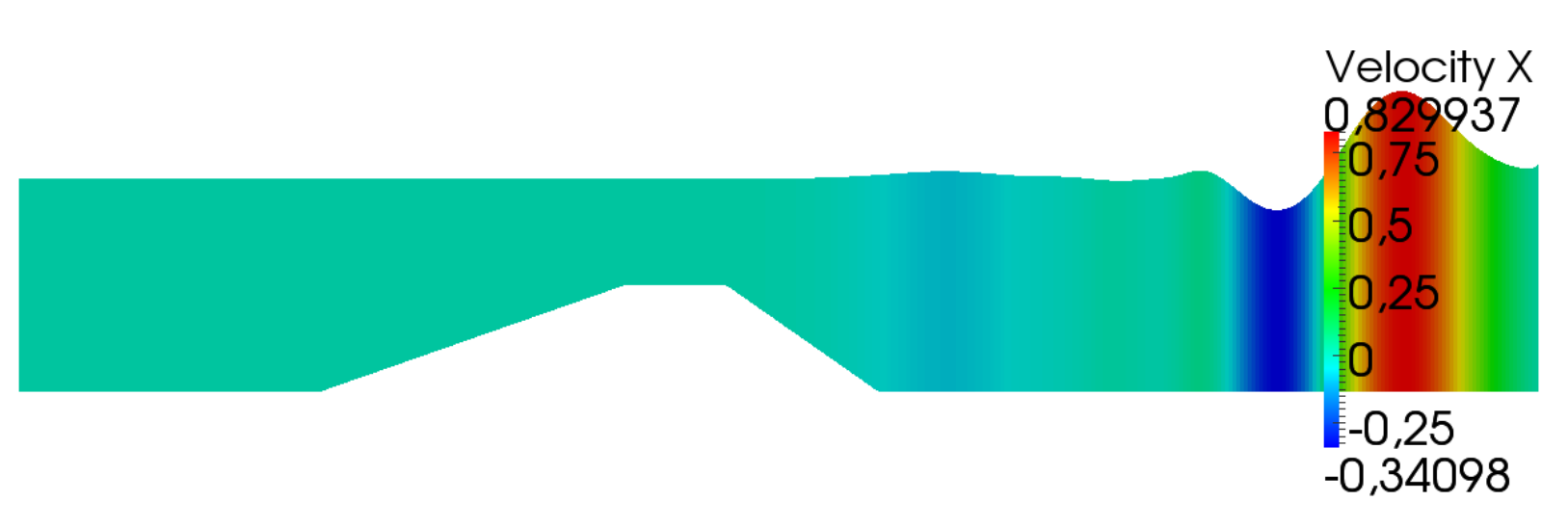} & \includegraphics[width=0.45\textwidth]{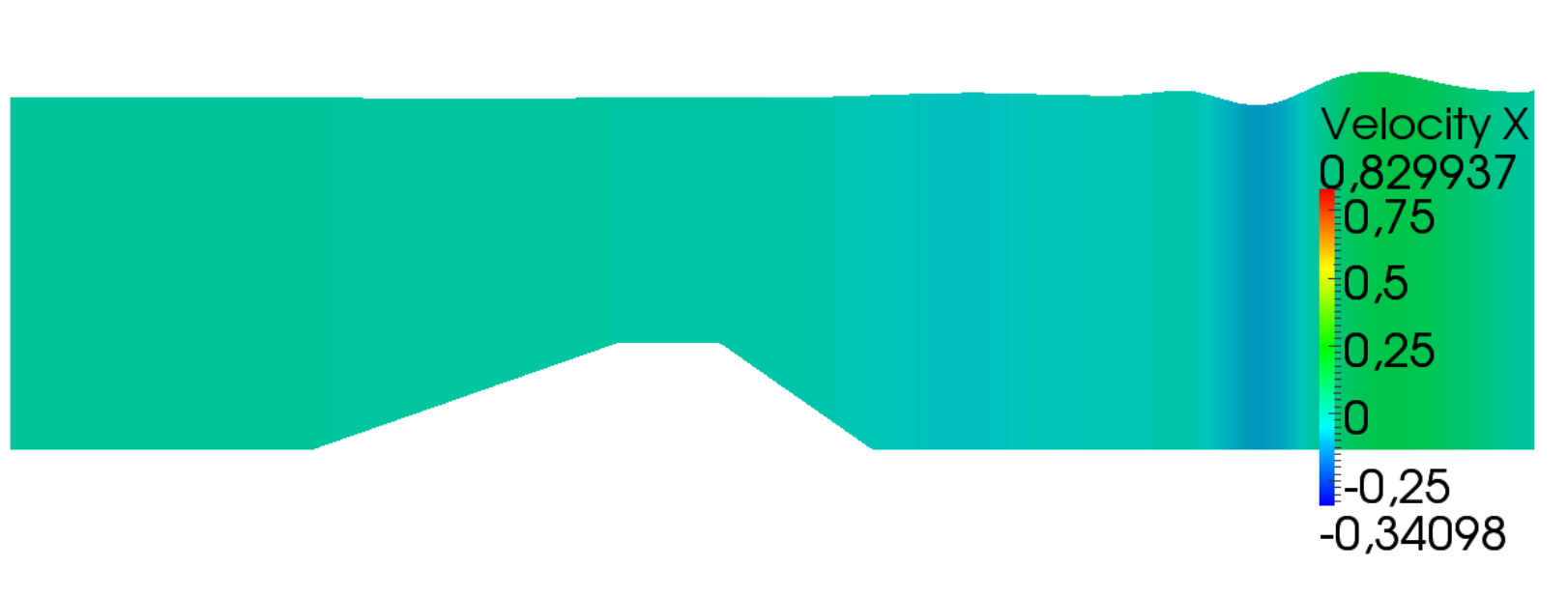}\\
 \includegraphics[width=0.45\textwidth]{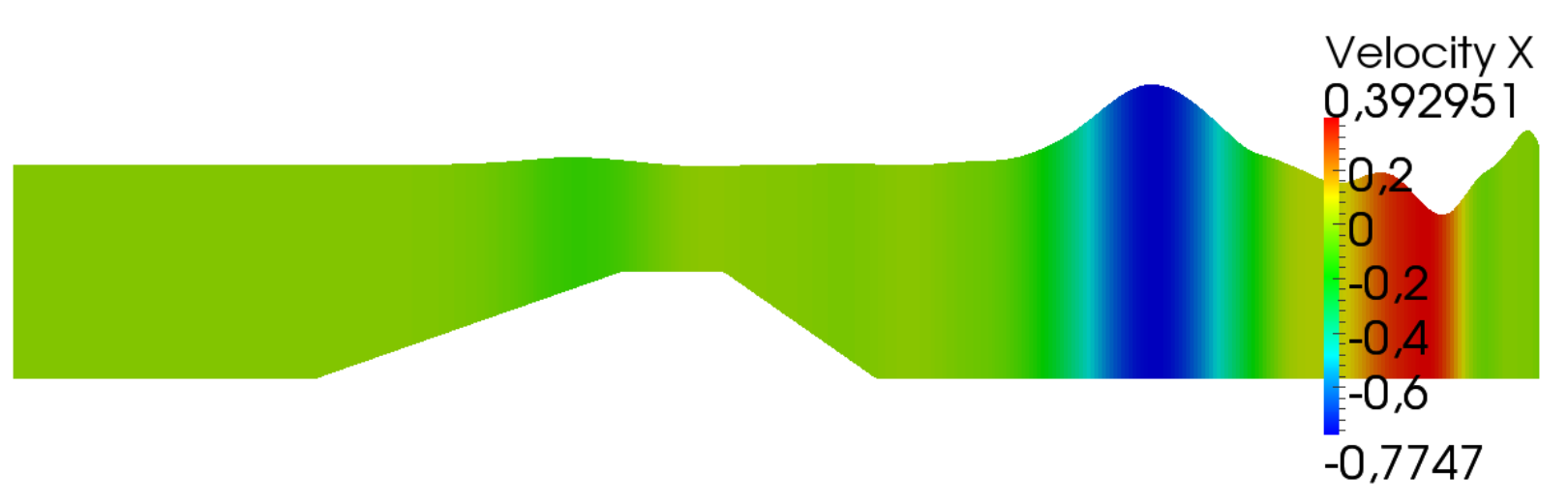} & \includegraphics[width=0.45\textwidth]{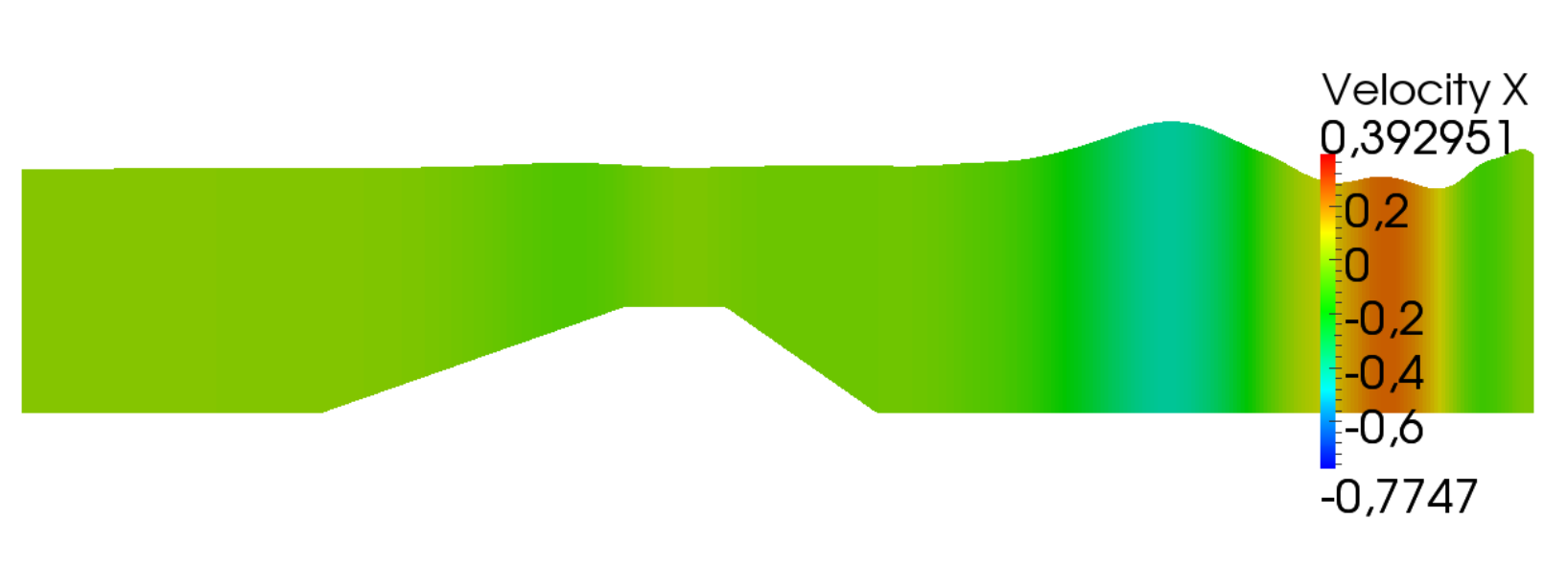}\\
 \includegraphics[width=0.45\textwidth]{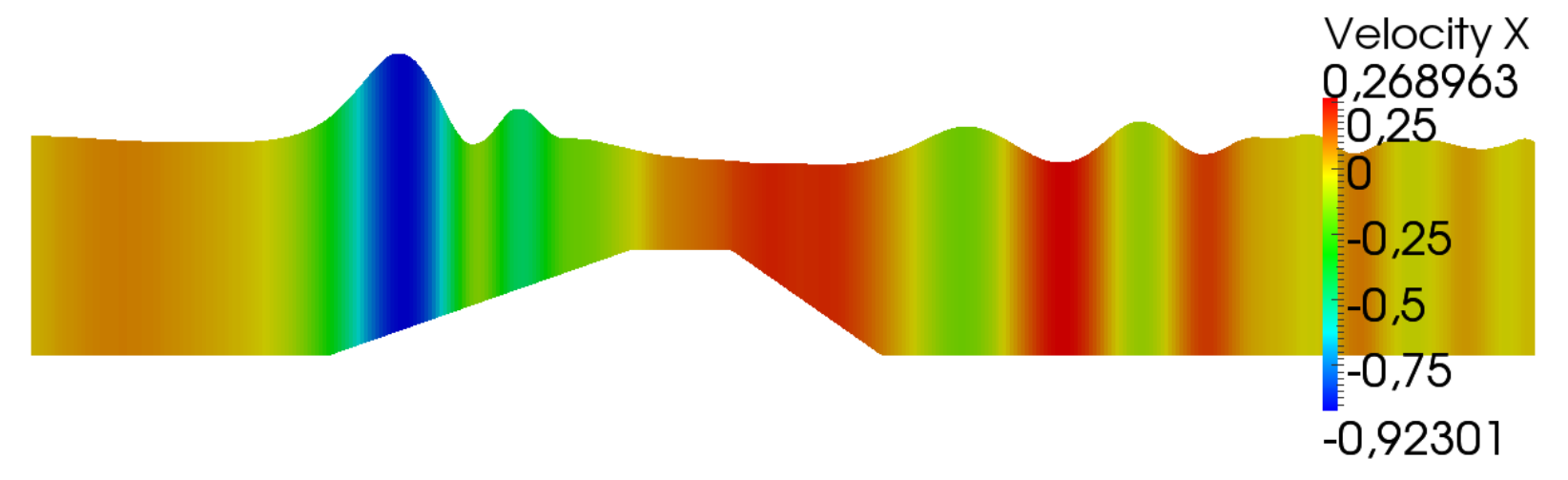} & \includegraphics[width=0.45\textwidth]{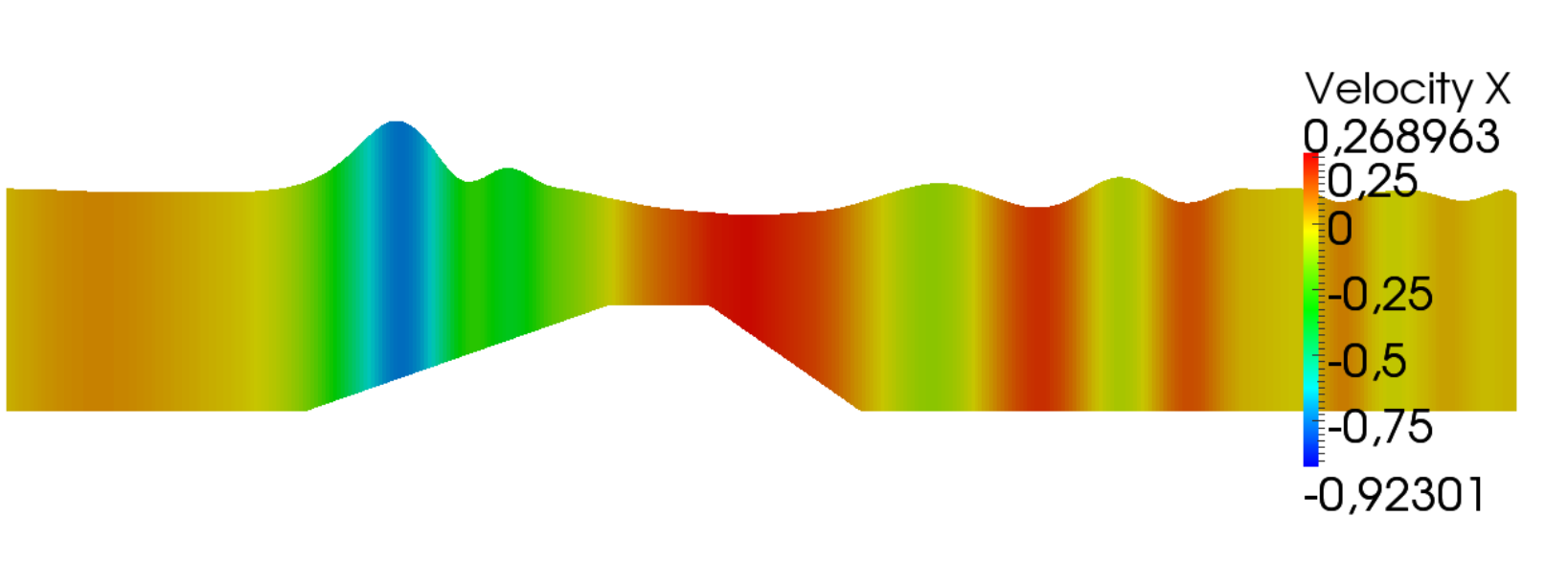}\\
\includegraphics[width=0.45\textwidth]{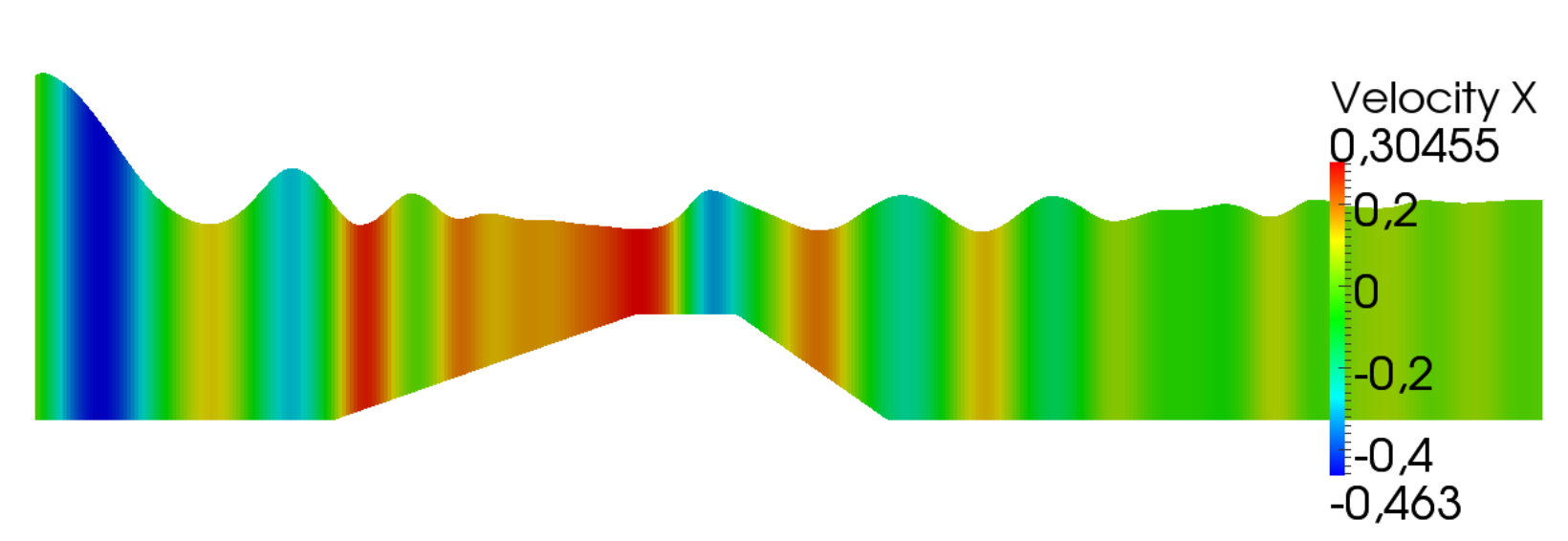} & \includegraphics[width=0.45\textwidth]{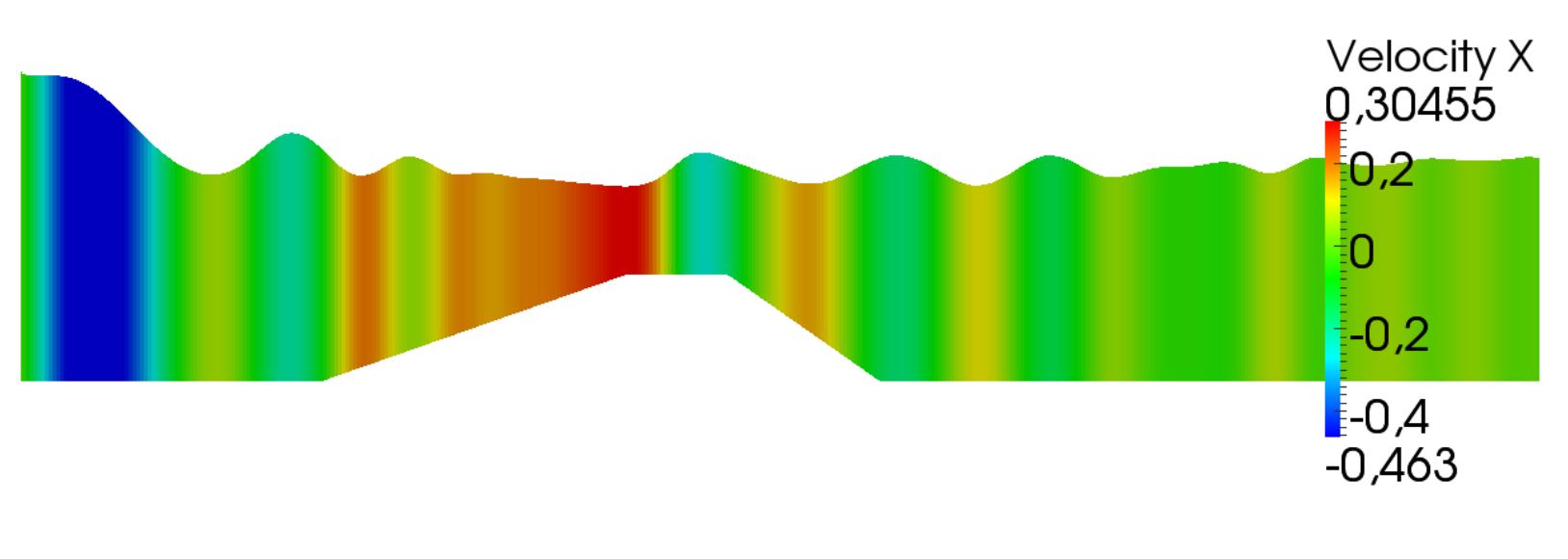} 
 \end{tabular}
\caption[Snapshots of assimilation on non-hydrostatic Saint-Venant system - horizontal velocity]{Snapshots of data assimilation on the non-hydrostatic Saint-Venant system, compared to the reference simulation for the horizontal velocity.}
\label{svtsnapshots}
\end{center}
\end{figure}

Once again, the observer shows good agreement with the observations. We propose in Figure~\ref{dingemanserror} a plot of the relative error between exact and complete observations in space, every $0.08$ s or $0.04$s. Even if this system is more general than the one studied theoretically in this article we can see that the observer built only from the knowledge of the water height is still efficient. The decrease of the error with $\lambda$ is obtained like for Thacker solutions. We may also notice that the time between two measurements influences the error, as it provides a better interpolation process.
\begin{figure}[htbp]
 \begin{center}
 \begin{tikzpicture}[scale=0.8]
\begin{axis}[
xlabel=$\lambda$,
ylabel=Relative error (\%),
title={Error at 20 s}]]
\addplot[color=red,mark=*] table[x=gain,y=error]
	 {FigDA/dingemans1.txt};
	 \addlegendentry{$\Delta t$ measures = 0.08}	
\addplot[color=blue,mark=*] table[x=gain,y=error]
	 {FigDA/dingemans2.txt};	
	 \addlegendentry{$\Delta t$ measures = 0.04}	 
\end{axis}
\end{tikzpicture}
  \caption[Error curve of assimilation on the non-hydrostatic Saint-Venant system]{Relative error $\frac{\norm{\widehat H - H}_{L^{1}}}{\norm{H}_{L^{1}}}$ for different values of $\lambda$ for the non-hydrostatic Saint-Venant system. Two observations time steps are tested.}
\label{dingemanserror}
 \end{center}
\end{figure}
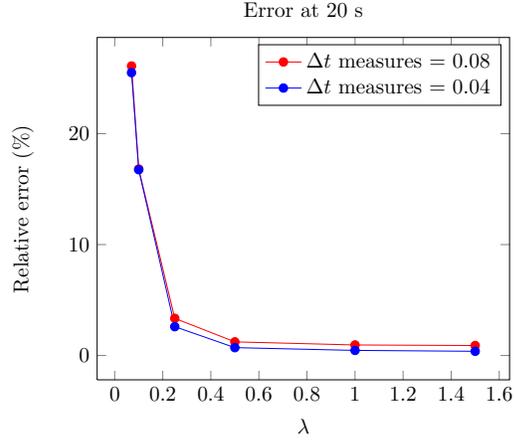

%% file: discussion.tex

\section{Discussion}
In this article, we propose a new data assimilation strategy adapted to hyperbolic conservation laws and of practical use in complex cases as illustrated with the Saint-Venant system with source terms.  As usually done when tackling a problem on hyperbolic balance laws, the scalar case allows us to develop an extensive theoretical framework to demonstrate the interest of our approach. The joint numerical experimentations proved the good agreement of the numerical results with the presented theory. From a theoretical point of view it remains to precise some results in the case of shocks and partial measurements but this work understands the theoretical foundations behind the definition of a Luenberger observer at the kinetic level. The simplicity of the resulting estimator should be then compared to the classical difficulties arising in data assimilation for conservation laws. Moreover, it is in the case of systems like the non-hydrostatic Saint-Venant system or even the free surface hydrostatic Navier-Stokes 
equations that the use of the kinetic representation shows a critical advantage with respect to classical Luenberger definition since it allows to design a new --~nevertheless natural --~observer on the macroscopic variables which can be demonstrated to be stable in terms of energy balance. Even if the theoretical framework remains to be fully defined in the case of systems, the numerical results are in fact very encouraging.

We believe that the perspectives of introducing the kinetic representations in a data assimilation context are now numerous from a theoretical point of view to the practical context. It also encourages to try to apply other sequential approaches with the kinetic formulation, like for example considering a Kalman filter to the "linear" kinetic representation and then contemplating the resulting observer on the macroscopic variables. Note that such strategy of considering an additional variable before applying the Kalman filter can circumvent the ''curse of dimensionality'' only if a resulting system depending on the macroscopic variable can eventually be retrieved. The same type of open question raises if the system is approximated by a --~potentially very large --~system of ODEs via the method of characteristics where, on each characteristics, a Kalman filter can be easily applied. An interesting challenge is then to prove that the large system of resulting ODEs --~each controlled by a Kalman filter --~can be recombined to present an original and \emph{computable} filter on the initial PDE system

%% file: bmps_final.bbl
\begin{thebibliography}{10}

\bibitem{bristeau1}
E.~Audusse, F.~Bouchut, M.-O. Bristeau, R.~Klein, and B.~Perthame.
\newblock A fast and stable well-balanced scheme with hydrostatic
  reconstruction for {Shallow Water} flows.
\newblock {\em SIAM J. Sci. Comput.}, 25(6):2050--2065, 2004.

\bibitem{Audusse2010a}
E.~Audusse, M.-O. Bristeau, M.~Pelanti, and J.~Sainte-Marie.
\newblock Approximation of the hydrostatic navier-stokes system for density
  stratified flows by a multilayer model. kinetic interpretation and numerical
  validation.
\newblock {\em J. Comp. Phys.}, 2010.

\bibitem{Audusse2010b}
E.~Audusse, M.-O. Bristeau, B.~Perthame, and J.~Sainte-Marie.
\newblock A multilayer saint-venant system with mass exchanges for shallow
  water flows. derivation and numerical validation.
\newblock {\em ESAIM: Mathematical Modelling and Numerical Analysis},
  45(01):169--200, 2011.

\bibitem{Auroux2005}
D.~Auroux and J.~Blum.
\newblock Back and forth nudging algorithm for data assimilation problems.
\newblock {\em Comptes Rendus Mathematique}, 340(12):873 -- 878, 2005.

\bibitem{auroux2012}
D.~Auroux and M.~Nodet.
\newblock The back and forth nudging algorithm for data assimilation problems :
  theoretical results on transport equations.
\newblock {\em ESAIM: Control, Optimisation and Calculus of Variations},
  18:318--342.

\bibitem{bardos2002formalism}
C.~Bardos and O.~Pironneau.
\newblock A formalism for the differentiation of conservation laws.
\newblock {\em Comptes Rendus Mathematique}, 335(10):839--845, 2002.

\bibitem{bardos2005data}
C.~Bardos and O.~Pironneau.
\newblock Data assimilation for conservation laws.
\newblock {\em Methods and Applications of Analysis}, 12(2):103--134, 2005.

\bibitem{Bellman1957}
R.E. Bellman.
\newblock {\em Dynamic {P}rogramming}.
\newblock Princeton University Press, 1957.

\bibitem{Bensoussan1971}
A.~Bensoussan.
\newblock {\em Filtrage optimal des syst{\`e}mes lin{\'e}aires}.
\newblock Dunod, 1971.

\bibitem{bertoglio2011}
C.~Bertoglio, D.~Chapelle, M.A. Fern\'andez, J-F. Gerbeau, and P.~Moireau.
\newblock State observers of a vascular fluid-structure interaction model
  through measurements in the solid.
\newblock {\em Comput. Meth. Appl. Mech. Engrg}, 256(2013):149--168, 2011.

\bibitem{Blum2008}
J.~Blum, F.-X. Le~Dimet, and I.~M. Navon.
\newblock Data assimilation for geophysical fluids.
\newblock In R.~Temam and J.~Tribbia, editors, {\em Handbook of Numerical
  Analysis: Computational Methods for the Atmosphere and the Oceans}. Elsevier,
  2008.

\bibitem{bouchut1999averaging}
F.~Bouchut and L.~Desvillettes.
\newblock Averaging lemmas without time fourier transform and application to
  discretized kinetic equations.
\newblock In {\em Proceedings of the Royal Society of Edinburgh-A-Mathematics},
  volume 129, pages 19--36. Cambridge Univ Press, 1999.

\bibitem{boulangerthese}
A.-C. Boulanger.
\newblock {\em Modeling, simulation and data assimilation about a hydrodynamics
  and biology coupling problem}.
\newblock PhD thesis, Paris VI University, 2013.

\bibitem{Bristeau201151}
M.-O. Bristeau, N~Goutal, and J.~Sainte-Marie.
\newblock Numerical simulations of a non-hydrostatic shallow water model.
\newblock {\em Computers and Fluids}, 47(1):51 -- 64, 2011.

\bibitem{Chapelle2011}
D.~Chapelle, N.~C\^{\i}ndea, M.~de~Buhan, and P.~Moireau.
\newblock Exponential convergence of an observer based on partial field
  measurements for the wave equation.
\newblock {\em Mathematical Problems in Engineering}, (581053):1--12, 2012.

\bibitem{Chapelle2012}
D.~Chapelle, N.~C\^{\i}ndea, and P.~Moireau.
\newblock Improving convergence in numerical analysis using observers - the
  wave-like equation case.
\newblock {\em Mathematical Models and Methods in Applied Sciences}, 22, 2012.
\newblock In press, doi:10.1142/S0218202512500406.

\bibitem{Mallet2012}
D.~Chapelle, M.~Fragu, V.~Mallet, and P.~Moireau.
\newblock Fundamental principles of data assimilation underlying the verdandi
  library: applications to biophysical model personalization within euheart.
\newblock {\em Medical \& Biological Eng \& Computing}, pages 1--13, 2012.

\bibitem{coron2009control}
J.-M. Coron.
\newblock {\em Control and nonlinearity}, volume 136.
\newblock American Mathematical Soc., 2009.

\bibitem{dafermos2005}
C.M. Dafermos.
\newblock {\em Hyperbolic conservation laws in continuum physics}, volume 325.
\newblock Springer Verlag, 2005.

\bibitem{golse1988regularity}
F.~Golse, P.-L. Lions, B.~Perthame, and R.~Sentis.
\newblock Regularity of the moments of the solution of a transport equation.
\newblock {\em Journal of functional analysis}, 76(1):110--125, 1988.

\bibitem{golse1985resultat}
F.~Golse, B.~Perthame, and R.~Sentis.
\newblock Un resultat de compacite pour les equations de transport et
  application au calcul de la valeur propre principale de i'operateur de
  transport cr acad sc.
\newblock {\em T}, 301:l, 1985.

\bibitem{Haine2011}
G.~Haine and K.~Ramdani.
\newblock {Reconstructing initial data using observers}.
\newblock {\em Numer Math}, pages 1--37, August 2011.

\bibitem{hoke1976initialization}
J.E. Hoke and R.A. Anthes.
\newblock The initialization of numerical models by a dynamic-initialization
  technique(fluid flow models for wind forecasting).
\newblock {\em Monthly Weather Review}, 104:1551--1556, 1976.

\bibitem{dimet2010variational}
F.-X. Le~Dimet and O.~Talagrand.
\newblock Variational algorithms for analysis and assimilation of
  meteorological observations: theoretical aspects.
\newblock {\em Tellus A}, 38(2):97--110, 2010.

\bibitem{lions1994kinetic}
P.-L. Lions, B.~Perthame, and E.~Tadmor.
\newblock A kinetic formulation of multi-dimensional scalar conservation laws
  and related equations.
\newblock {\em Journal of the American Mathematical Society}, 7(1):169--192,
  1994.

\bibitem{Liu1997}
K.~Liu.
\newblock Locally distributed control and damping for the conservative systems.
\newblock {\em SIAM J. Control Optim.}, 35(5):1574--1590, 1997.

\bibitem{Luenberger1963}
D.G. Luenberger.
\newblock {\em Determining the State of a Linear with Observers of Low Dynamic
  Order}.
\newblock PhD thesis, Stanford University, 1963.

\bibitem{Luenberger1971}
D.G. Luenberger.
\newblock An introduction to observers.
\newblock {\em IEEE Transactions on Automatic Control}, 16:596--602, 1971.

\bibitem{Moireau2009}
P.~Moireau, D.~Chapelle, and P.~Le~Tallec.
\newblock Filtering for distributed mechanical systems using position
  measurements: Perspectives in medical imaging.
\newblock {\em Inverse Problems}, 25(3):035010 (25pp), 2009.

\bibitem{Navon2009}
I.M. Navon.
\newblock Data assimilation for numerical weather prediction : a review.
\newblock In S.K. Park and L.~Xu, editors, {\em Data Assimilation for
  Atmospheric, Oceanic, and Hydrologic Applications}, volume XVIII. Springer,
  2009.

\bibitem{perthame1990boltzmann}
B.~Perthame.
\newblock Boltzmann type schemes for gas dynamics and the entropy property.
\newblock {\em SIAM Journal on Numerical Analysis}, 27(6):1405--1421, 1990.

\bibitem{Perthame2002}
B.~Perthame.
\newblock {\em Kinetic formulation of conservation laws}.
\newblock Oxford lecture series in mathematics and its applications. Oxford
  University Press, 2002.

\bibitem{perthame2001kinetic}
B.~Perthame and C.~Simeoni.
\newblock A kinetic scheme for the saint-venant system with a source term.
\newblock {\em Calcolo}, 38(4):201--231, 2001.

\bibitem{Ramdani2010}
K.~Ramdani, M.~Tucsnak, and G.~Weiss.
\newblock {Recovering the initial state of an infinite-dimensional system using
  observers}.
\newblock {\em Automatica}, 46(10):1616--1625, 2010.

\bibitem{saintemarievertical2009}
J.~sainte Marie.
\newblock Vertically averaged models for the free surface non-hydrostatic euler
  system: Derivation and kinetic interpretation.
\newblock {\em Mathematical Models and Methods in Applied Sciences},
  21(03):459--490, 2011.

\bibitem{stauffer1993}
D.R. Stauffer, R.~David, and J.-W. BAO.
\newblock Optimal determination of nudging coefficients using the adjoint
  equations.
\newblock {\em Tellus A}, 45(5):358--369, 1993.

\bibitem{stauffer1990use}
D.R. Stauffer and N.L. Seaman.
\newblock Use of four-dimensional data assimilation in a limited-area mesoscale
  model. part i: Experiments with synoptic-scale data.
\newblock {\em Monthly Weather Review}, 118(6):1250--1277, 1990.

\bibitem{thacker1981some}
W.~C. Thacker.
\newblock Some exact solutions to the nonlinear shallow-water wave equations.
\newblock {\em Journal of Fluid Mechanics}, 107(1):499--508, 1981.

\bibitem{vidard2003optimal}
A.~Vidard, F.-X. Le~Dimet, and A.~Piacentini.
\newblock Optimal determination of nudging coefficients.
\newblock {\em Tellus Series A: Dynamic meteorology and oceanography},
  55(1):1--15, 2003.

\bibitem{zou1992optimal}
X.~Zou, I.M. Navon, and F.-X. LeDimet.
\newblock An optimal nudging data assimilation scheme using parameter
  estimation.
\newblock {\em Quarterly Journal of the Royal Meteorological Society},
  118(508):1163--1186, 1992.

\end{thebibliography}
